\documentclass[a4paper,11pt]{article}
\title{Capturing nonclassical shocks in nonlinear elastodynamic with a conservative finite volume scheme}
\author{Nina Aguillon\thanks{nina.aguillon@math.u-psud.fr, Université Aix Marseille. This work has been carried out in the framework of the Labex Archimède (ANR-11-LABX-0033) and of the A*MIDEX project (ANR-11-IDEX-0001-02), funded by the ``Investissements d'Avenir" French Government programme managed by the French National Research Agency (ANR). The authors also thanks the Université Paris Sud where this work began.}}

\usepackage[english]{babel}
\usepackage{verbatim}
\usepackage{fancyvrb}
\usepackage[utf8x]{inputenc}
\usepackage{multicol,lipsum,float}
\usepackage[T1]{fontenc}
\usepackage{graphicx}
\usepackage{multicol}
\usepackage{graphics}
\usepackage{xcolor}
\usepackage{pgf}
\usepackage{amssymb}
\usepackage{amsmath}
\usepackage{amsfonts}
\usepackage{amssymb}
\usepackage{graphicx}
\usepackage{bm}
\usepackage{amsthm}
\usepackage{enumerate}
\usepackage{textcomp}
\usepackage{mathrsfs}
\usepackage{caption}
\usepackage{stmaryrd}
\usepackage{url}
\usepackage{psfrag}
\usepackage{hyperref}

\newtheorem{thm}{Theorem}[section]

\newtheorem{prop}[thm]{Proposition}

\theoremstyle{remark}
\newtheorem{rem}[thm]{Remark}

\theoremstyle{definition}
\newtheorem{defi}[thm]{Definition}

\newcommand{\N}{\mathbb{N}}
\newcommand{\R}{\mathbb{R}}

\newcommand{\Z}{\mathbb{Z}}
\newcommand{\eps}{\epsilon}

\begin{document}
\maketitle

\begin{abstract}
For a model of nonlinear elastodynamics, we construct a finite volume scheme which is able to capture nonclassical shocks (also called undercompressive shocks). Those shocks verify an entropy inequality but are not admissible in the sense of Liu. They verify a kinetic relation which describes the jump, and keeps an information on the equilibrium between a vanishing dispersion and a vanishing diffusion. The scheme presented here is by construction exact when the initial data is an isolated nonclassical shock. In general, it does not introduce any diffusion near shocks, and hence nonclassical solutions are correctly approximated. The method is fully conservative and does not use any shock-tracking mesh. This approach is tested and validated on several test cases. In particular, as the nonclassical shocks are not diffused at all, it is possible to obtain large time asymptotics.
\end{abstract}

\bigskip



\section*{Introduction}
Hyperbolic systems of conservation laws often arise as limits of systems of partial differential equations including small scales effects, like diffusion and dispersion, when the parameters driving the small scales effects vanish. For example, the compressible Euler equations are the (formal) limit of the compressible Navier--Stokes equations, when the diffusion parameter tends to zero. However, the solutions of the hyperbolic system do not inherit all the properties of the augmented system, and it particular they are not always smooth. More importantly, in the class of weak solutions, the Cauchy problem for the hyperbolic system may admit infinitely many solutions, plenty of them being non relevant toward the augmented system. Thus it is necessary to add some criterions that select the solutions that are indeed limits, as the parameters vanish, of the solutions of the  augmented diffusive-dispersive system. One of this criterion yields to the selection of \emph{classical solutions}, which roughly speaking corresponds to the case where diffusion is dominant. But when two small scale effects are of comparable strength (say diffusion and dispersion) \emph{nonclassical solutions} appear. They contain shocks that do not verify the classical criterion but do arise as limits of the diffusive-dispersive system.

Let us sketch a general framework for nonclassical solutions of hyperbolic systems. Consider the Cauchy problem for a system of conservation laws
\begin{equation} \label{eq:hypsys} 
\begin{cases}
 \partial_{t} U(t,x) + \partial_{x} f(U(t,x)) =0, \\
 U(t=0,x)=U^{0}(x),
\end{cases}
\end{equation}
where $t$ is positive, $x$ belongs to $\R$, $U$ is the vector regrouping the $N \geq 1$ unknowns belonging to a open convex set $\Omega$ of $\R^{N}$, and $f= \R^{N} \rightarrow \R^{N}$ is the flux function, which we assume to be smooth. We suppose that System~\eqref{eq:hypsys} is strictly hyperbolic, i.e. that the Jacobian matrix of $f$ is diagonalizable with simple eigenvalues $\lambda_{1}(U)< \cdots < \lambda_{n}(U)$. We denote by $r_{i}(U)$ a right eigenvector associated to the eigenvalue~$\lambda_{i}(U)$.

Solutions of hyperbolic systems are in general not smooth, even when the initial data is smooth. Discontinuities appear in finite time and it is necessary to consider weak solutions. Doing so, uniqueness is lost, and an additional criterion has to be imposed to select a unique solution. For example, the vector of unknown $U$ is asked to verify, in addition to the initial System~\eqref{eq:hypsys}, a so called entropy inequality
\begin{equation} \label{eq:entropyineq}
 \partial_{t} \mathcal{U}(U) + \partial_{x} \mathcal{W}(U) \leq 0.
\end{equation}
 The entropy $\mathcal{U}$ and the entropy flux $\mathcal{W}$ usually come from considerations on an augmented version of the system under study, where small physical terms like diffusion and dispersion are not neglected. 
 Such an augmented system has additional properties, for example the total energy is conserved. Inequality~\eqref{eq:entropyineq} states that at the limit, this energy is dissipated. In general, from the augmented system only one entropy inequality is deduced, thus the solutions of~\eqref{eq:hypsys} are not expected to verify other entropy inequalities than~\eqref{eq:entropyineq}.  When the system is genuinely nonlinear, i.e. when 
\begin{equation} \label{eq:GNL}
 \forall i \in \{1, \cdots, n\} , \quad \forall U \in \Omega, \quad \nabla \lambda_{i}(U) \cdot r_{i}(U) \neq 0,
\end{equation}
 and when the entropy $\mathcal{U}$ is stricly convex, the entropy inequality~\eqref{eq:entropyineq} contains enough information to select only one weak solution of~\eqref{eq:hypsys}. In the sequel we suppose that the entropy $\mathcal{U}$ is strictly convex. In  the general case where~\eqref{eq:GNL} does not hold, the addition of a single entropy inequality is not sufficient to select a unique weak solution. 
 
 The question of uniqueness of the weak solution is strongly related to the choice of a criterion to decide wether a shock is admissible or not. Two states $U_{L}$ and $U_{R}$ of $\Omega$ are linked by a shock if there exists a real $s$, called the speed of the shock, so that the Rankine--Hugoniot relations
 \begin{equation} \label{eq:RH}
 s(U_{L}-U_{R})= f(U_{L})-f(U_{R})
\end{equation}
hold. The shock is entropy satisfying if it verifies~\eqref{eq:entropyineq} in a weak sense, i.e. if 
\begin{equation} \label{eq:dissipation}
s (\mathcal{U}(U_{L})-\mathcal{U}(U_{R})) -(\mathcal{W}(U_{L})-\mathcal{W}(U_{R}))   \leq 0.
\end{equation}

When the system is genuinely nonlinear (i.e. when~\eqref{eq:GNL} holds),~\eqref{eq:dissipation} is equivalent to the Lax inequalities~\cite{Lax57}
\begin{equation} \label{eq:LaxCrit}
\exists i \in \{1, \cdots, n\}: \, \lambda_{i}(U_{R}) < s < \lambda_{i}(U_{L}) , \, s<\lambda_{i+1}(U_{R}) \ \text{ and } \  s>\lambda_{i-1}(U_{L}),
\end{equation}
which themselves are an extension to systems ($N>1$) of Oleinik's criterion~\cite{O57}.

In~\cite{Liu74}, Liu generalized this criterion to the non genuinely nonlinear case where $\nabla \lambda_{i} \cdot r_{i}$ may vanish.  Liu proved that his criterion selects a unique selfsimilar solution to the Riemann problem~(\ref{eq:hypsys}--\ref{eq:entropyineq}). We recall that a Riemann problem is the case of a piecewise constant initial data
$$ U^{0}(x)= U_{L} \mathbf{1}_{x<0} + U_{R} \mathbf{1}_{x>0}, \qquad (U_{L},U_{R}) \in \Omega^{2}. $$
 For systems that are not genuinely nonlinear, the Liu criterion is stronger than~\eqref{eq:dissipation} (whereas it is equivalent for genuinely nonlinear systems). An admissible shock in the sense Liu always verifies~\eqref{eq:dissipation}, but some shocks may verify~\eqref{eq:dissipation} without satisfying the Liu criterion. These shocks are called nonclassical. Moreover, unlike in the genuinely nonlinear case, the addition of the single entropy inequality~\eqref{eq:entropyineq} is not enough to regain uniqueness of the weak solution, because too many shocks are allowed. In other words, the inequality entropy~\eqref{eq:entropyineq} does not contain enough information on the small scale effects in the augmented system.  Uniqueness can be regain by strongly constraining the nonclassical shocks to verify an algebraic relation of the form
%
$$ U_{L}= \Phi^{\flat}(U_{R}). $$
The function $\Phi^{\flat}$ is called a \emph{kinetic relation}. 

It has been proved that the addition of such a kinetic relation selects a unique weak solution to the Riemann problem for several models, including models of phase transitions in solids in~\cite{AK91}, \cite{LT01} and~\cite{LT02} or in liquids~\cite{LT00}, \cite{HL00},~\cite{SY95} and a model of magnetohydrodynamics~\cite{HL00}. The link between kinetic relations and augmented diffusive dispersive equations is explored in~\cite{JKS95},~\cite{LeFloch} and~\cite{HL97} in the scalar case and in~\cite{S89} for gas dynamics with a Van der Waals pressure law. The literature on the subject is so wide that we only cited a few key references. The interested reader will find a more comprehensive bibliography on this topic in~\cite{LeFloch}.

The numerical approximation of nonclassical solutions is challenging, because they are by nature very sensitive to the equilibrium between the small scale effects. The stake is to preserve this equilibrium at the numerical level.  Roughly speaking, usual finite volume schemes introduce a numerical viscosity which destroys the equilibrium between the small scales effects. The diffusion becomes dominant and the schemes converge toward the classical solution, even though they are based on nonclassical Riemann solvers.

To overcome this difficulty, two types of schemes have been proposed. On the one hand, it is possible to use high order schemes consistent with the augmented system (see for example \cite{HL98}, \cite{LMR02}, \cite{LR00} and~\cite{KR10}). As outlined in~\cite{HL98}, it is necessary to discretize the flux of the hyperbolic part with a high enough order, once again to keep the exact balance between the small scales. On the other hand, some schemes relies on the tracking of nonclassical shocks (see for example~\cite{CL03}, \cite{CG08}, \cite{BCLL08}, \cite{P11}, \cite{CCER12} and~\cite{CDMG} ). In that case the kinetic relation is taken into account the scheme, typically through the use of an exact nonclassical Riemann solver. 
 
The conservative finite volume scheme presented here  belongs to this last category. It is built to be exact when the initial data is  a nonclassical shock. In general, it does not introduce any numerical diffusion near nonclassical shocks (a particular class of nonclassical solutions) and turns out to capture correctly nonclassical solutions.  This scheme extends to hyperbolic systems the discontinuous reconstruction scheme introduced in the scalar case in~\cite{BCLL08} and recently used in~\cite{CDMG}. 

In the first section, we present the hyperbolic system admitting nonclassical solution for which we construct the scheme, namely the model of nonlinear elasticity studied in~\cite{LT01} and~\cite{HL00}. The second section is devoted to the construction of the scheme itself. We present a way to select cells in which a nonclassical shock is reconstructed, explain how this shock is reconstructed and give the numerical flux. Eventually, we prove that the scheme is exact when the initial data is a nonclassical shock. In the third and last section, we propose several test cases involving nonclassical shocks. It appears that the proposed scheme capture nonclassical shocks very sharply. Let us outline that the scheme is fully conservative (unlike Glimm type schemes~\cite{CL03}) and uses a fixed grid, which allows the computation of solutions containing several interacting nonclassical shocks.

\textbf{Acknowledgment} The author warmly thanks Frédéric Lagoutière for his support and advice.
 
\section{The Riemann problem for a nonlinear elasticity model}
This paper is devoted to the numerical approximation of the solutions of the system of conservation laws
\begin{equation} \label{eq:NLelasticity} 
\begin{cases}
 \partial_{t} v - \partial_{x} \sigma(w) = 0, \\
 \partial_{t} w - \partial_{x} v = 0, \\
 v(t=0,x)=v^{0}(x), \\
 w(t=0,x)=w^{0}(x),
\end{cases}
\end{equation}
where the stress $\sigma$ is twice differentiable, and verifies
\begin{equation} \label{eq:propstress}
 w \sigma''(w) >0, \, \sigma'(w)>0 \ \text{ and } \ \ \lim_{|w| \rightarrow + \infty} \sigma'(w)= + \infty.
\end{equation}
We are interested in weak solutions of~\eqref{eq:NLelasticity} that are morally limits when $\eps$ tends to $0^{+}$ of the augmented system
\begin{equation} \label{eq:ANLelasticity} 
\begin{cases}
 \partial_{t} v^{\eps} - \partial_{x} \sigma(w^{\eps}) = \eps \partial_{xx} v + \alpha \eps^{2} \partial_{xxx} w, \\
 \partial_{t} w - \partial_{x} v = 0. 
\end{cases}
\end{equation}
The parameter $\alpha$ is positive. From~\eqref{eq:ANLelasticity} we deduce the following entropy inequality for System~\eqref{eq:NLelasticity} (see~\cite{LeFloch}):
\begin{equation} \label{eq:SEI} 
\partial_{t} \left( \frac{v^{2}}{2} + \int_{0}^{w} \sigma(z) dz \right) + \partial_{x} \left( -v \sigma(w) \right) \leq 0.
\end{equation}

In this section we briefly recall the results of Thanh and LeFloch~\cite{LT01} on the Riemann problem for System~\eqref{eq:NLelasticity}, i.e. the case where
\begin{equation} \label{eq:RPElasto}
 \begin{cases}
 v^{0}(x)  &= v_{L} \mathbf{1}_{x<0} + v_{R} \mathbf{1}_{x>0}, \\
 w^{0}(x) &= w_{L} \mathbf{1}_{x<0} + w_{R} \mathbf{1}_{x>0} .
\end{cases}
\end{equation}
 We adopt the notation of this paper. It is recalled on Figure~\ref{F:Notations}.
 
\begin{psfrags}
\psfrag{f}{$w \mapsto \sigma(w)$}
\psfrag{w}{$w$}
\psfrag{w'}{$w'$}
\psfrag{0}{$0$}
\psfrag{Phi1}{$\Phi^{\natural}(w)$}
\psfrag{Phi2}{$\Phi^{\flat}_{\infty}(w)$}
\psfrag{Phi3}{$\Phi^{-\natural}(w)$}
\begin{figure} 
\centering
 \includegraphics[width=0.8 \linewidth, clip=true, trim=3cm 1cm 3cm 1cm]{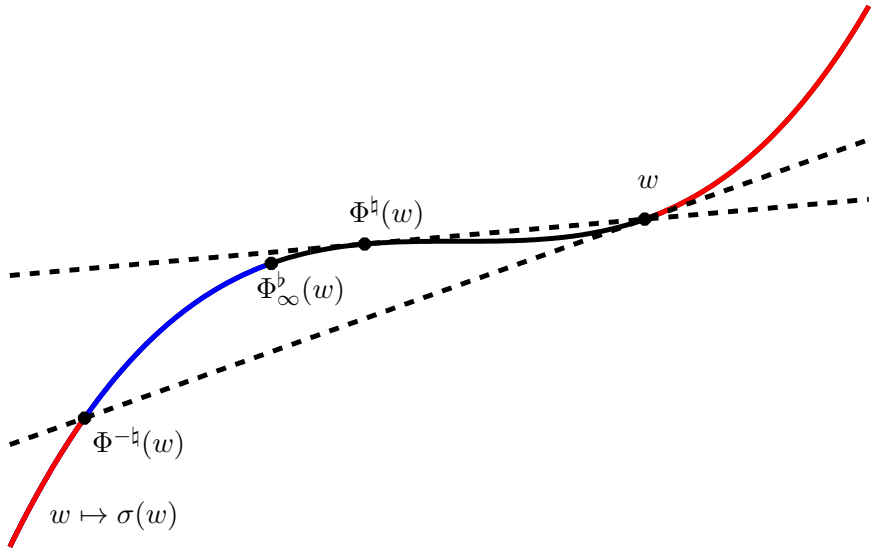}
 \caption{Notations used for solving the Riemann problem. $\Phi^{-\natural}(w)$ is the first Liu shock reachable from $w$. The entropy dissipation between $w$ and $\Phi_{\infty}^{\flat}$ is zero. The states in red can be linked to $w$ by a Liu shock, the blue ones by a nonclassical shock.} \label{F:Notations}
\end{figure}
\end{psfrags}

If the  left state $(v_{L},w_{L})$ and the right state $(v_{R},w_{R})$ are linked by a $1$-shock, its speed is equal to $-s(w_{L},w_{R})$, where
\begin{equation} \label{eq:shockspeed}
 s(w_{L}, w_{R}) =  \sqrt{\frac{\sigma(w_{L})-\sigma(w_{R})}{w_{L}-w_{R}}},
\end{equation}
and $v_{R}$ is given by
\begin{equation} \label{eq:H1}
 v_{R}= v_{L} +s(w_{L},w_{R}) (w_{R}-w_{L}):= \mathcal{H}_{1}(w_{R},v_{L},w_{L}). 
\end{equation}
If they are linked by a $2$-shocks, its speed is $+s(w_{L},w_{R})$, and 
$$ v_{L}= v_{R} -s(w_{L},w_{R}) (w_{L}-w_{R}):= \mathcal{H}_{2}(w_{L},v_{R},w_{R}). $$
Graphically, the speed of a shock $s(w_{L},w_{R})$ corresponds to the slope of the segment joining $(w_{L}, \sigma(w_{L}))$ to $(w_{R}, \sigma(w_{R}))$.

Let us now focus on when a $1$-shock is admissible or not, either in the sense of the entropy dissipation~\eqref{eq:SEI} or in the sense of Liu. The results for $2$-shocks are the same, with the roles of $w_{L}$ and $w_{R}$ reversed in what follows. A careful study of Equation~\eqref{eq:SEI} applied to the shock yields that it verifies the entropy inequality if and only if
$$w_{L}^{2} \leq w_{L} w_{R} \quad \text{ or } \quad w_{R} w_{L} \leq \Phi^{\flat}_{\infty}(w_{L}) w_{L},$$
where $\Phi^{\flat}_{\infty}(w_{L})$ is the real having the opposite sign than $w_{L}$ for which the entropy dissipation~\eqref{eq:dissipation} vanishes. The function $\Phi_{\infty}^{\flat}$ is its own inverse.

For System~\eqref{eq:NLelasticity}, a shock verifies the Liu criterion if and only if for all $w$ between $w_{L}$ and $w_{R}$,
$$ \frac{\sigma(w)- \sigma(w_{R})}{w-w_{R}} \geq \frac{\sigma(w_{L})- \sigma(w_{R})}{w_{L}-w_{R}}. $$
Geometrically, it means that the segment joining $(w_{L},\sigma(w_{L}))$ to $(w_{R},\sigma(w_{R}))$ is above the graph of $\sigma$ if $w_{R}>w_{L}$, and below if $w_{R}<w_{L}$, see Figure~\ref{F:Notations}. It follows that there exists an invertible function $\Phi^{\natural}$ such that a shock verifies the Liu criterion if and only if 
$$w_{L}^{2} \leq w_{L} w_{R} \quad \text{ or } \quad w_{R} w_{L} \leq \Phi^{-\natural}(w_{L}) w_{L}.$$
The fact that System~\eqref{eq:NLelasticity} is not genuinely nonlinear yields 
$$w_{L} \Phi^{-\natural}(w_{L}) < w_{L} \Phi^{\flat}_{\infty}(w_{L}).$$ 
 Thus, if $w_{R}$ lies in between $\Phi^{\flat}_{\infty}(w_{L})$ and $\Phi^{-\natural}(w_{L})$, the shock between $(v_{L}, w_{L})$ and  $(v_{R}, w_{R})$ dissipates the entropy~\eqref{eq:SEI} but is not admissible in the sense of Liu.  Such shocks are called \emph{nonclassical shocks}.

Many Riemann problems for System~(\ref{eq:NLelasticity}-\ref{eq:SEI}) admit an infinity of solutions (see~\cite{LT01}, Section $3$). Uniqueness can be obtained by imposing a \emph{kinetic relation} that strongly constrain the nonclassical shocks by imposing
 \begin{equation}\label{eq:KR}
 \begin{cases}
 w_{L}= \Phi^{\flat, 1}(w_{R}) & \text{for a $1$-nonclassical shock}, \\
 w_{R}= \Phi^{\flat, 2}(w_{L}) & \text{for a $2$-nonclassical shock},
\end{cases}
\end{equation}
where $\Phi^{\flat, 1}$ et $\Phi^{\flat, 2}$ both verify
$$ w \Phi^{\flat}_{\infty}(w)  \leq w \Phi^{\flat, i}(w) \leq w \Phi^{\natural}(w) $$
but are not necessarily equal. We now state the main result of~\cite{LT01}.

\begin{thm} \label{thm:TL}[Thanh, LeFloch] If the kinetic functions $\Phi^{\flat,1}$ and $\Phi^{\flat,2}$ are monotone decreasing, the Riemann Problem~(\ref{eq:NLelasticity}-\ref{eq:SEI}-\ref{eq:KR}) has a unique selfsimilar solution.
 \end{thm}
The exact Riemann solver associated to Theorem~\ref{thm:TL} is the foundation of the scheme presented in the next section. Moreover, we will use the inner structure of the Riemann solution. Thus for the sake of completeness, we describe below the forward $1$-wave and the backward $2$-wave. 
\begin{prop}\label{p:FL1}[Thanh, LeFloch] Let $(v_{L},w_{L})$ be a fixed left state, and $(v_{M},w_{M})$ be a state linked to $(v_{L},w_{L})$ by a $1$-wave. Then the nature of this $1$-wave is the following:
\begin{itemize}
 \item if $w_{L} w_{M} > w_{L}^{2}$, it is a classical shock;
 \item if $0 \leq w_{L} w_{M} \leq w_{L}^{2}$, it is a rarefaction wave;
 \item if $w_{L} \varphi^{-\flat, 1}(w_{L}) < w_{L} w_{M} \leq 0$, it first contains a $1$-rarefaction linking $(v_{L},w_{L})$ to the state $(\mathcal{H}_{1}(\varphi^{\flat, 1}(w), v_{M}, w_{M}), \varphi^{\flat, 1}(w_{M}))$, which is itself linked to $(v_{M},w_{M})$ by a $1$-nonclassical shock;
 \item if $w_{L} w_{M}< w_{L} \varphi^{-\flat, 1}(w_{L})$, two cases arise:
 \begin{itemize}
 \item if $-s(w_{L}, \varphi^{\flat, 1}(w_{M}))<-s(\varphi^{\flat, 1}(w_{M})),w_{M})$, it first contains a $1$-classical shock linking $(v_{L},w_{L})$ to the state $(\mathcal{H}_{1}(\varphi^{\flat, 1}(w), v_{M}, w_{M}), \varphi^{\flat, 1}(w_{M}))$, which is itself linked to $(v_{M},w_{M})$ by a $1$-nonclassical shock;
 \item otherwise, it is just a $1$-classical shock (in which $w$ changes sign).
\end{itemize}
\end{itemize}
 \end{prop}
 
 \begin{prop} \label{p:BL2}[Thanh, LeFloch] Let $(v_{R},w_{R})$ be a fixed right state, and $(v_{M},w_{M})$ be a state such that $(v_{M},w_{M})$ and $(v_{R},w_{R})$ are linked by a $2$-wave. Then the nature of this $2$-wave is the following:
\begin{itemize}
 \item if $w_{R} w_{M} > w_{R}^{2}$, it is a classical shock;
 \item if $0 \leq w_{R} w_{M} \leq w_{R}^{2}$, it is a rarefaction wave;
 \item if $w_{R} \varphi^{-\flat, 2}(w_{R}) < w_{R} w_{M} \leq 0$, it first contains a $2$-nonclassical shock linking $(v_{M},w_{M})$ to the state $(\mathcal{H}_{2}(\varphi^{\flat, 2}(w), v_{M}, w_{M}), \varphi^{\flat, 1}(w_{M}))$, which is itself linked to $(v_{R},w_{R})$ by a $2$-rarefaction wave;
 \item if $w_{R} w_{M}< w_{R} \varphi^{-\flat, 2}(w_{R})$, two cases arise:
 \begin{itemize}
 \item if $s(w_{M}, \varphi^{\flat, 2}(w_{M}))<s(\varphi^{\flat, 2}(w_{M})),w_{R})$, it first contains a $2$-nonclassical shock linking $(v_{M},w_{M})$ to the state $(\mathcal{H}_{2}(\varphi^{\flat, 2}(w), v_{M}, w_{M}), \varphi^{\flat, 1}(w_{M}))$, which is itself linked to $(v_{R},w_{R})$ by a $2$-classical shock;
 \item otherwise, it is just a $2$-classical shock (in which $w$ changes sign).
\end{itemize}
\end{itemize}
 \end{prop}
 The distinction between the last two cases in the last point of the enumerations of Proposition~\ref{p:FL1} and~\ref{p:BL2} can be reexpressed geometrically, as illustrated on Figure~\ref{F:PhiDiese}.
\begin{prop} \label{p:PhiDiese}
 The exists a function $\varphi^{\sharp}$ such that for $i \in \{1, 2\}$, for all real $w$ and for all real $w_{M}$ such that $w w_{M}< w \varphi^{-\flat, i}(w)$, the quantity $|s(w, \varphi^{\flat, i}(w_{M}))|$ is larger than $|s(\varphi^{\flat, 1}(w_{M})),w_{M})|$ if and only if $w \varphi^{\flat, i}(w_{M})$ is larger than $w \varphi^{\sharp}(w,w_{M})$. In particular if $0 \geq w w_{M} \geq w_{M} \varphi^{\sharp}(w,w_{M})$, the solution is a classical shock (in which $w$ changes sign).
\end{prop}

\begin{psfrags}
\psfrag{wR}{$w_{R}$}
\psfrag{wM}{$w_{M}$}
\psfrag{wB}{$\Phi^{\flat,2}(w_{M})$}
\psfrag{Phi}{$\Phi^{\sharp}(w_{R},w_{M})$}
 \begin{figure}[h!tp]
 \includegraphics[width=0.53\linewidth, clip=true, trim=1cm 1cm 1cm 1cm]{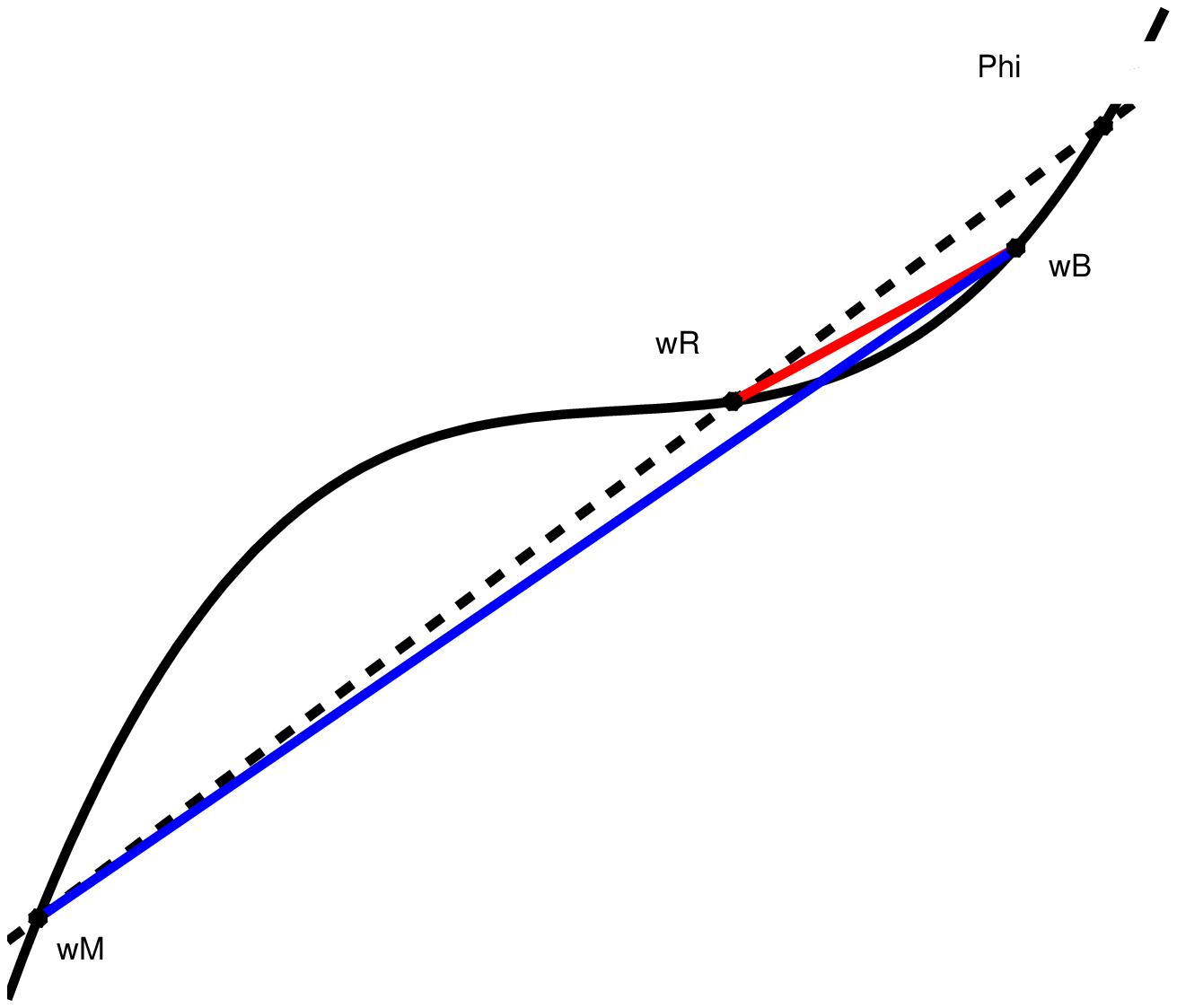}
 \includegraphics[width=0.47\linewidth, clip=true, trim=1cm 1cm 1cm 1cm]{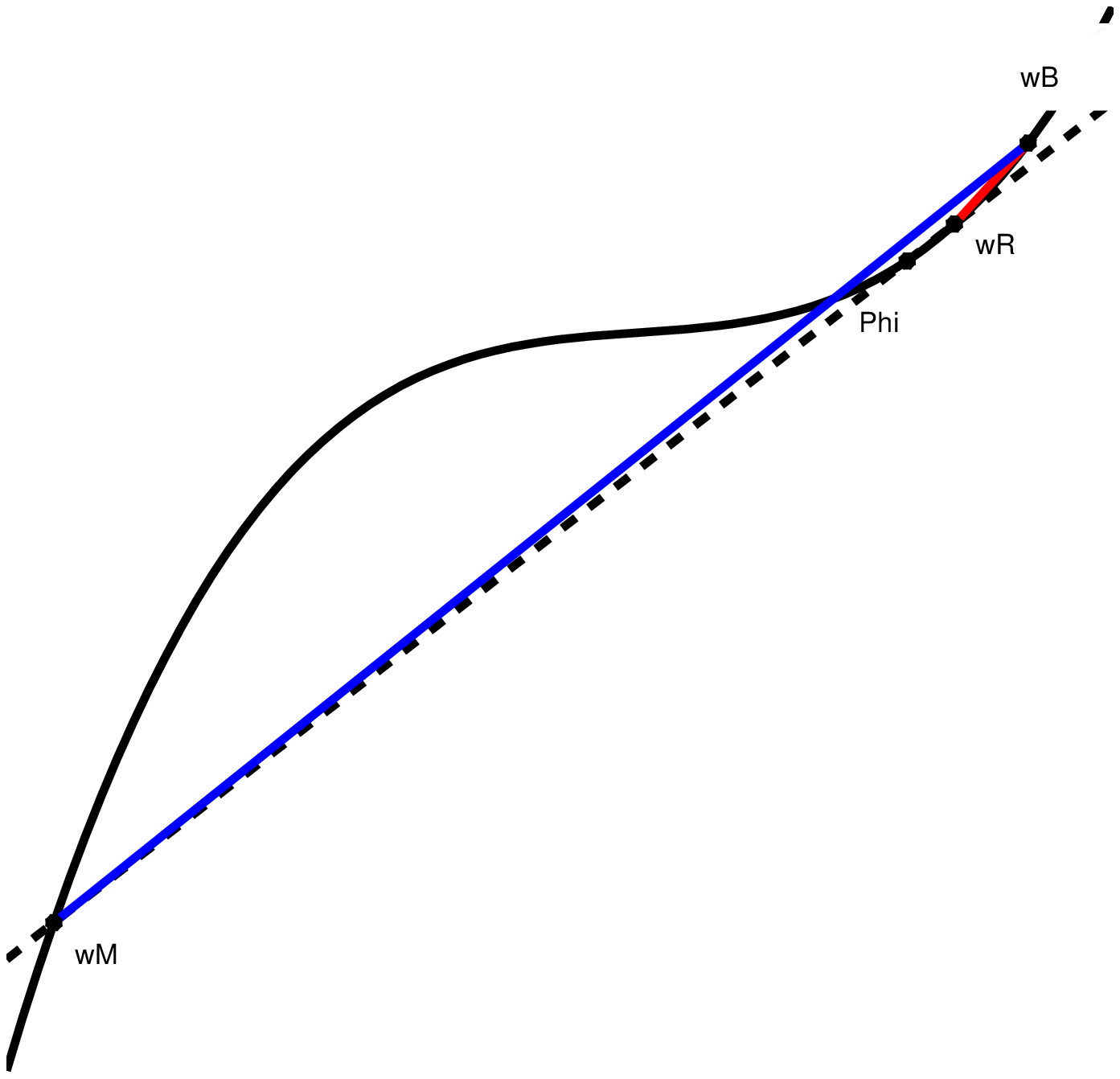}
\caption{The solution contains a $2$-nonclassical shock  (in blue) followed by $2$-classical shock  (in red) if their speeds are correctly ordered, i.e if the slope of the blue segment is smaller than the red one (right). This depends on the relative position of $\Phi^{\flat,2}(w_{M})$ and $\Phi^{\sharp}(w_{M},w_{L})$.} \label{F:PhiDiese}
\end{figure}
\end{psfrags}

\section{The scheme}
We now describe the scheme. The time discretization is denoted by $t^{0}=0 < t^{1} < \cdots < t^{n} < \cdots$, and at each time $t^{n}$, the space is discretized with cells having all the same size $\Delta x$.  We denote by $x_{j-1/2}^{n}$ their extremities and by $x_{j}^{n}$ there centers. The centers of the cells will move at each time step but every cell will always be of size $\Delta x$. The speed of the mesh is  defined by $V^{n}_{\textup{mesh}}= \frac{x_{j}^{n+1}-x_{j}^{n}}{t^{n+1}-t^{n}}$. Integrating the first line of~\eqref{eq:NLelasticity} over the set
$$\{(t,x), t^{n} \leq t < t^{n+1}, x_{j-1/2}^{n}+ V_{\textup{mesh}}^{n} t <x < x_{j+1/2}^{n}+ V_{\textup{mesh}}^{n} t \} $$
we obtain the integral formula
$$ 
\begin{aligned}
 \int_{x_{j-1/2}^{n}}^{x_{j+1/2}^{n}}  v(t^{n+1}, x) dx =&  \int_{x_{j-1/2}^{n}}^{x_{j+1/2}^{n}} v(t^{n}, x) dx  \\
 &- \left[\int_{t^{n}}^{t^{n+1}} (-\sigma(w)-V_{\textup{mesh}}^{n} v)(s, x_{j+1/2}^{n} + V_{\textup{mesh}}^{n} (s-t^{n})) ds  \right.\\
 & \left. -  \int_{t^{n}}^{t^{n+1}} (-\sigma(w)-V_{\textup{mesh}}^{n} v)(s, x_{j-1/2}^{n} + V_{\textup{mesh}}^{n} (s-t^{n})) ds \right].
\end{aligned}
$$
The formula is identical for the variable $w$ with the flux $(-v-V_{\textup{mesh}}^{n} w)$ instead of $(-\sigma(w)-V_{\textup{mesh}}^{n} v$. The principle of  finite volume  methods is based on this integral formula. A finite volume scheme writes
\begin{equation} \label{eq:FVS}
\begin{cases}
 v_{j}^{n+1} &= v_{j}^{n} -  \frac{t^{n+1}-t^{n}}{\Delta x} ( f_{j+1/2}^{n,v}- f_{j-1/2}^{n,v}), \\
 w_{j}^{n+1} &= w_{j}^{n} -  \frac{t^{n+1}-t^{n}}{\Delta x} ( f_{j+1/2}^{n,w}- f_{j-1/2}^{n,w}),
\end{cases}
\end{equation}
 where $v_{j}^{n}$ plays the role of the mean value 
 $$\frac{1}{\Delta x} \int_{x_{j-1/2}^{n}}^{x_{j+1/2}^{n}} v(t^{n}, x) dx $$
 and $f_{j+1/2}^{n,v}$ plays the role of the flux on the edge $x=x_{j-1/2}+ V_{\textup{mesh}}^{n}(s-t^{n})$
 $$ \frac{1}{t^{n+1}-t^{n}} \int_{t^{n}}^{t^{n+1}} (-\sigma(w)-V_{\textup{mesh}}^{n} v)(s, x_{j+1/2}^{n} + V_{\textup{mesh}}^{n} (s-t^{n})) ds  $$
 (and similarly for $w_{j}^{n}$ and $f_{j+1/2}^{n,w}$). The scheme is initialized with the exact average of the initial data
 \begin{equation} \label{eq:inisamp}
\forall j \in \Z, \   v_{j}^{0}= \frac{1}{\Delta x} \int_{x_{j-1/2}^{0}}^{x_{j+1/2}^{0}} v^{0}(x) dx \quad \text{ and } \quad  w_{j}^{0}= \frac{1}{\Delta x} \int_{x_{j-1/2}^{0}}^{x_{j+1/2}^{0}} w^{0}(x) dx.
\end{equation}
A particular finite volume scheme is characterized by a particular choice of a formula expressing the numerical fluxes $f_{j+1/2}^{n,v}$ and $f_{j+1/2}^{n,w}$ in terms of the $(v_{k}^{n})_{k \in \Z}$ and $(w_{k}^{n})_{k \in \Z}$. Our choice is described in the next sections.
 
  Recall that our aim is to derive a scheme which is exact when the initial data is a nonclassical shock. In that case, the initial sampling~\eqref{eq:inisamp} introduces a small amount of dissipation (unless if by miracle the shock initially falls on an interface). The numerical initial data contains an intermediate value that does not correspond to any pointwise value of $(v^{0},w^{0})$.  More generally, at the end of every time step, finite volume schemes contains a~$L^{1}$-projection on the mesh. Details of the solution are lost in that step, and intermediate values are created in the shocks profile. The main idea of the discontinuous reconstruction scheme is to rebuilt, from the mean values $(v_{j}^{n},w_{j}^{n})_{j \in \Z}$, an initial data that contains more details.  This idea is also the foundation of well-known other schemes, like the MUSCL scheme~\cite{VL79} and its central version~\cite{NT90}. Our reconstruction consists in adding nonclassical shocks in some cells in which we are able to detect that a nonclassical shock where lying inside the cell before the $L^{1}$-projection.  Thus the reconstruction will be very precise near nonclassical shocks, and not on the smooth parts of the solution, taking the opposite strategy of~\cite{VL79} and~\cite{NT90}
 
Following this idea, our scheme can be decomposed in three elementary steps.
\begin{itemize}
\item First, we detect the special cells  which are associated with nonclassical shocks. This detection step is described in Section~\ref{SDetect}.
\item Then, we reconstruct nonclassical shocks inside those particular cells, in the sense that we replace the mean values $(v_{j}^{n}, w_{j}^{n})$ by piecewise constant functions having the form of a nonclassical shock. This procedure is explained in Section~\ref{SRec}.
\item Eventually, the numerical fluxes are computed by letting the reconstructed nonclassical shocks evolve during the time step $t^{n+1}-t^{n}$. The use of a moving mesh makes this computation easy, see Section~\ref{SAdv}.
\end{itemize}
In the last section, we extend the scheme to reconstruct classical shocks as well.

\subsection{Detection of nonclassical shocks} \label{SDetect}

The key idea of the scheme is to see, whenever it is possible, each mean value $(v_{j}^{n}, w_{j}^{n})$ produced by the finite volumes scheme as the average of some nonclassical shock located somewhere inside the cell. Of course, where $(v_{j-1}^{n}, w_{j-1}^{n})$ and $(v_{j+1}^{n}, w_{j+1}^{n})$ are linked by a nonclassical shock, it is the one that should be reconstructed. In general, $(v_{j-1}^{n}, w_{j-1}^{n})$ and $(v_{j+1}^{n}, w_{j+1}^{n})$ are not linked by a single nonclassical shock, but by the full pattern of two waves, each of them likely to contain a succession of a classical shock or a rarefaction followed by a nonclassical shock (see Propositions~\ref{p:FL1} and~\ref{p:BL2}). In that case, we reconstruct one of the nonclassical shock appearing in the solution of the Riemann problem. It is chosen thanks to the following Lemma.

\begin{prop} \label{PDetectPT}
 If the states $(v_{L},w_{L})$ and $(v_{R},w_{R})$ are linked by a $1$-nonclassical shock, then
 \begin{equation} \label{eq:Detect1}
 w_{L} w_{R}<0 \  , \  (w_{L}-w_{R})(v_{L}-v_{R}) >0  \text{ and } w_{L} w_{R} \geq w_{R} \varphi^{\sharp}(w_{L},w_{R}).
\end{equation}
 If they are linked by a $2$-nonclassical shock, then
 $$w_{L} w_{R}<0 \ (w_{L}-w_{R})(v_{L}-v_{R}) <0 \ \text{ and } \ w_{L} w_{R} \geq w_{L} \varphi^{\sharp}(w_{L},w_{R}). $$
\end{prop}
\begin{proof}
The condition $w_{L} w_{R}<0$ always holds through a nonclassical shock. The second one is a straightforward consequence of the Rankine--Hugoniot relations~\eqref{eq:RH}. In the case of nonlinear elastodynamics~\eqref{eq:NLelasticity}, they write
 \begin{equation} \label{eq:RHElasto}
 \begin{cases}
 (-1)^{i} s(w_{L},w_{R})(v_{L}-v_{R})=\sigma(w_{R})- \sigma(w_{L}); \\
 (-1)^{i}s(w_{L},w_{R})(w_{L}-w_{R})=v_{R}-v_{L}.
\end{cases}
\end{equation}
with $i=1$ for $1$-shocks and $i=2$ for $2$-shocks. The positive real $s(w_{L},w_{R})$ is defined in Formula~\eqref{eq:shockspeed}.  Eventually, the last condition is a reminder of Proposition~\ref{p:PhiDiese}.
\end{proof}

\begin{defi} \label{DDesired}
 For all integer $j$ and all positive integer $n$, we denote by $(v_{j,\star}^{n}, w_{j,\star}^{n})$ the intermediate state appearing in the Riemann problem with left state $(v_{j-1}^{n}, w_{j-1}^{n})$ and right state $(v_{j+1}^{n}, w_{j+1}^{n})$. The \emph{left and right desired reconstructed states}, denoted respectively by $(\bar{v}_{j,L}^{n}, \bar{w}_{j,L}^{n})$ and $(\bar{v}_{j,R}^{n}, \bar{w}_{j,R}^{n})$, are defined as follows.
 \begin{itemize}
 \item If $w_{j-1}^{n} w_{j+1}^{n}<0$,  $(w_{j-1}^{n}-w_{j+1}^{n})(v_{j-1}^{n}-v_{j+1}^{n}) >0$, $w_{j-1}^{n} w_{j+1}^{n} \geq w_{j+1}^{n} \varphi^{\sharp}(w_{j-1}^{n},w_{j+1}^{n})$ and $w_{j-1}^{n} w_{j,\star}^{n}<0$, the solution contains a $1$-shock in which $w$ changes sign.
 \begin{itemize}
 \item If this shock is nonclassical, it is preceded by a classical wave and we set
 $$ 
\begin{cases}
(\bar{v}_{j,L}^{n}, \bar{w}_{j,L}^{n}) &= (\mathcal{H}_{1}(\Phi^{ \flat,1}(w_{j,\star}^{n}), v_{j,\star}^{n}, w_{j,\star}^{n}), \Phi^{ \flat,1}(w_{j,\star}^{n})), \\
(\bar{v}_{j,R}^{n}, \bar{w}_{j,R}^{n}) &= (v_{j,\star}^{n}, w_{j,\star}^{n}).
\end{cases} 
$$
 \item  If this shock is classical, it is the only contribution in the $1$-wave and we set
 $$ 
\begin{cases}
(\bar{v}_{j,L}^{n}, \bar{w}_{j,L}^{n}) &= (v_{j-1}^{n}, w_{j-1}^{n}), \\
(\bar{v}_{j,R}^{n}, \bar{w}_{j,R}^{n}) &= (v_{j,\star}^{n}, w_{j,\star}^{n}).
\end{cases} 
$$
\end{itemize}
  \item If $w_{j-1}^{n} w_{j+1}^{n}<0$,  $(w_{j-1}^{n}-w_{j+1}^{n})(v_{j-1}^{n}-v_{j+1}^{n}) <0$, $w_{j-1}^{n} w_{j+1}^{n} \geq w_{j-1}^{n} \varphi^{\sharp}(w_{j-1}^{n},w_{j+1}^{n})$ and $w_{j+1}^{n} w_{j,\star}^{n}<0$, the solution contains a $2$-shock in which $w$ changes sign.
 \begin{itemize}
 \item If this shock is nonclassical, it is followed by a classical wave and we set
 $$ 
\begin{cases}
(\bar{v}_{j,L}^{n}, \bar{w}_{j,L}^{n}) &=(v_{j,\star}^{n}, w_{j,\star}^{n}),  \\
(\bar{v}_{j,R}^{n}, \bar{w}_{j,R}^{n}) &= (\mathcal{H}_{2}(\Phi^{ \flat,2}(w_{j,\star}^{n}), v_{j,\star}^{n}, w_{j,\star}^{n}), \Phi^{ \flat,2}(w_{j,\star}^{n})).
\end{cases} 
$$
 \item If this shock is classical, it is the only contribution in the $2$-wave and we set
 $$ 
\begin{cases}
(\bar{v}_{j,L}^{n}, \bar{w}_{j,L}^{n}) &= (v_{j,\star}^{n}, w_{j,\star}^{n}), \\
(\bar{v}_{j,R}^{n}, \bar{w}_{j,R}^{n}) &= (v_{j+1}^{n}, w_{j+1}^{n}).
\end{cases} 
$$
\end{itemize}
\item In the other cases, we do not detect any relevant nonclassical shock and we set
 $$ 
\begin{cases}
(\bar{v}_{j,L}^{n}, \bar{w}_{j,L}^{n}) &= (v_{j}^{n}, w_{j}^{n}), \\
(\bar{v}_{j,R}^{n}, \bar{w}_{j,R}^{n}) &= (v_{j}^{n}, w_{j}^{n}).
\end{cases} 
$$
\end{itemize}
\end{defi}
\begin{rem}
 If there is a nonclassical shock in the Riemann problem with left state $(v_{j-1}^{n}, w_{j-1}^{n})$ and right state $(v_{j+1}^{n}, w_{j+1}^{n})$, then the first two three conditions holds by Proposition~\ref{PDetectPT}. Note that they are numerically very easy to check. The last condition insures that the Riemann problem indeed contains the expected shock; we verify it only on the cells that passed the first three tests.
\end{rem}

\subsection{Reconstruction} \label{SRec}
Once a nonclassical shock has been detected, it is placed inside its cell by conservation of the variables $v$ and $w$. We reverse the averaging step of finite volume schemes by replacing the mean value $v_{j}^{n}$ by the piecewise constant function
$$v_{\textup{rec}}^{n}(x)=  \bar{v}_{j,L}^{n} \mathbf{1}_{x<x_{j-1/2}+d_{j}^{n,v}} +   \bar{v}_{j,R}^{n} \mathbf{1}_{x>x_{j-1/2}+d_{j}^{n,v}}$$
with
\begin{equation}\label{eq:dv}
 d_{j}^{n,v} =\Delta x \frac{v_{j}^{n}-\bar{v}_{j,R}^{n}}{\bar{v}_{j,L}^{n}-\bar{v}_{j,R}^{n}}.
\end{equation}
If $d_{j}^{n,v}$ belongs to $(0, \Delta x)$ we have
$$ \int_{x_{j-1/2}^{n}}^{x_{j+1/2}^{n}} v_{\textup{rec}}^{n}(x) \, dx = \Delta x v_{j}^{n} $$
and no mass is loss when replacing the mean value by $v_{\textup{rec}}^{n}$ inside the $j$-th cell.

Reasoning similarly for the $w$ variable, we replace $w_{j}^{n}$ by
$$w_{\textup{rec}}^{n}(x)=  \bar{w}_{j,L}^{n} \mathbf{1}_{x<x_{j-1/2}+d_{j}^{n,w}} +   \bar{w}_{j,R}^{n} \mathbf{1}_{x>x_{j-1/2}+d_{j}^{n,w}},$$
with
\begin{equation}\label{eq:dw}
 d_{j}^{n,w} =\Delta x \frac{w_{j}^{n}-\bar{w}_{j,R}^{n}}{\bar{w}_{j,L}^{n}-\bar{w}_{j,R}^{n}}.
\end{equation}
We now state one trivial but crucial property of this reconstruction procedure. 
\begin{prop} \label{p:recexa} If  $(v_{j-1}^{n}, w_{j-1}^{n})$ and $(v_{j+1}^{n}, w_{j+1}^{n})$ are linked by a single nonclassical shock, and if it exists $\alpha$ in $(0,1)$ such that
$$ (v_{j}^{n}, w_{j}^{n})= \alpha (v_{j-1}^{n}, w_{j-1}^{n}) + (1-\alpha) (v_{j+1}^{n}, w_{j+1}^{n}), $$
 then 
$$(\bar{v}_{j,L}^{n}, \bar{w}_{j,L}^{n})= (v_{j-1}^{n}, w_{j-1}^{n}), \quad (\bar{v}_{j,R}^{n}, \bar{w}_{j,R}^{n})=(v_{j+1}^{n}, w_{j+1}^{n})  \quad \text{ and } \quad d_{j}^{n,v}=d_{j}^{n,w}= \alpha \Delta x.$$
\end{prop}
In particular if the initial data is a nonclassical shock, we have $v^0_{\textup{rec}}=v^0$ and  $w^0_{\textup{rec}}=w^0$: the numerical diffusion introduced by the initial sampling~\eqref{eq:inisamp} is cancelled by the reconstruction.

In general the two distances $d_{j}^{n,v}$ and $d_{j}^{n,w}$ are different, and it is possible that at least one of these distances does not belong to $(0, \Delta x)$. In that case we cancel the reconstruction, considering that seeing $(v_{j}^{n}, w_{j}^{n})$ as the mean value of the detected nonclassical shock is not relevant.
\begin{defi} \label{DRecStates}
 The \emph{left and right reconstructed states} are defined by:
 \begin{equation} \label{eq:recL}
 (v_{j,L}^{n}, w_{j,L}^{n})= 
\begin{cases}
  (\bar{v}_{j,L}^{n}, \bar{w}_{j,L}^{n}) & \text{ if } d_{j}^{n,v} \in (0, \Delta x) \text{ and } d_{j}^{n,w} \in (0, \Delta x), \\
 (v_{j}^{n},w_{j}^{n}) & \text{ otherwise, } 
\end{cases}
\end{equation}
\begin{equation} \label{eq:recR}
 (v_{j,R}^{n}, w_{j,R}^{n})= 
\begin{cases}
  (\bar{v}_{j,R}^{n}, \bar{w}_{j,R}^{n}) & \text{ if } d_{j}^{n,v} \in (0, \Delta x) \text{ and } d_{j}^{n,w} \in (0, \Delta x), \\
 (v_{j}^{n},w_{j}^{n}) & \text{ otherwise. } 
\end{cases}
\end{equation}
\end{defi}
According to the Rankine--Hugoniot relations~\eqref{eq:RH}, the nonclassical shock reconstructed in cell $j$ has  velocity
$$ s_{j}^{n} = \frac{v_{j,R}^{n}- v_{j,L}^{n}}{w_{j,L}^{n}- w_{j,R}^{n}} $$ 
and we set arbitrary $s_{j}^{n}$ to $\sqrt{\sigma'(w_{j})}$ when no reconstruction is performed (which reads $w_{j,L}^{n}= w_{j,R}^{n}=w_{j}^{n}$). 

\subsection{Advection of the reconstructed discontinuities} \label{SAdv}
The fluxes are computed by letting the reconstructed nonclassical shocks evolve during the time step and by computing exactly what goes through the interfaces. However, two discontinuities reconstructed in adjacent cells can interact and the waves resulting from the interaction can meet the line $x=x_{j+1/2}^{n}$ within the time step. It follows that if we want to use a fixed grid ($x_{j}^{n+1}=x_{j}^{n}$, or equivalently $V^{n}_{\textup{mesh}}=0$), the flux along the interface $x=x_{j+1/2}^{n}$ cannot be computed without resolving the wave interaction, which is obviously extremely costly. 

This can been avoiding by using a moving mesh. Let us recall that by moving mesh, we mean that the centers of the cells are moving from time to time, but that their size remains constant equals to $\Delta x$. Thus the numerical difficulties of handling cells with different and varying widths is avoided. The mesh  speed is chosen such that it is larger than the maximum of the waves speed:
\begin{equation} \label{eq:Vmesh}
 |V_{\textup{mesh}}^{n}| > V_{waves}^{n}= \max_{j \in \Z} \sqrt{ \sigma'(w_{j}^{n})},
\end{equation}
and the time step such that a wave cannot cross more than an entire cell during the time step:
\begin{equation} \label{eq:CFL}
 t^{n+1}-t^{n} \leq \frac{\Delta x}{ |V_{\textup{mesh}}^{n}|+V_{waves}^{n}}.
\end{equation}

These two hypothesis insure that any wave created at time $t^{n}$ inside the $j$-th cell can only cross the right interface $x=x_{j+1/2}^{n}+V_{\textup{mesh}}^{n} t$ if $V_{\textup{mesh}}^{n}<0$ and the left interface $x=x_{j-1/2}^{n}+V_{\textup{mesh}}^{n} t$ if $V_{\textup{mesh}}^{n}>0$. It also imply that if two waves interact inside the $j$-cell during the time step, the waves resulting from their interaction will not have time to catch up the left or the right interfaces. This is depicted on Figure~\ref{FMM}. Thus, computing the flux along the $j+1/2$-th interface $x=x_{j+1/2}^{n}+V_{\textup{mesh}}^{n} t$ boils down to compute the time at which the nonclassical shock reconstructed in cell $j$ crosses the interface if $V_{\textup{mesh}}^{n}<0$  (in cell $j+1$ otherwise). On Figure~\ref{FMM}, we also see that if no reconstruction is performed in cells $j$ and $j+1$, the flux is trivial to compute and is given by a simple left or right decentering (depending on the sign of $V_{\textup{mesh}}^{n}$), which exactly corresponds to the staggered Lax--Friedrichs flux.
\begin{psfrags}
 \psfrag{xj-1/2}{$x_{j-1/2}^{n}$}
 \psfrag{xj+1/2}{$x_{j+1/2}^{n}$}
 \psfrag{tn}{$t^{n}$}
 \psfrag{tn+1}{$t^{n+1}$}
 \psfrag{j-1}{$j-1$}
 \psfrag{j}{$j$}
 \psfrag{T}{$T_{j+1/2}^{n}$}
\begin{figure}
 \includegraphics[width=0.9\linewidth]{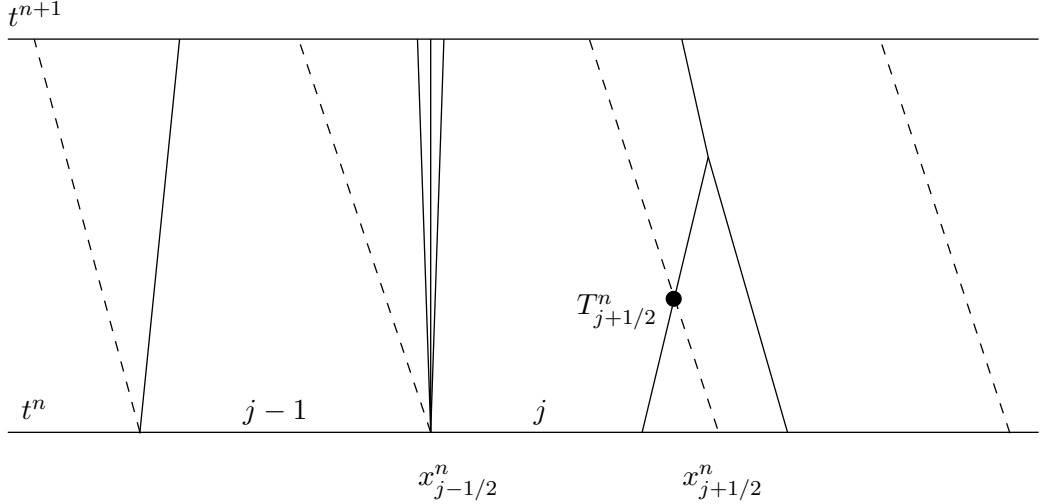}
 \caption{Computing the flux through the dashed interfaces can be done without solving waves interactions (interface $j+1/2$), and is trivial when no reconstruction is performed (cell $j-1$)} \label{FMM}
\end{figure}
\end{psfrags}

We now compute the flux in details, when $V_{\textup{mesh}}^{n}$ is negative and verifies~\eqref{eq:Vmesh}, like in Figure~\ref{FMM}.  Then the nonclassical shock reconstructed in cell $j$ may only cross the right interface $x=x_{j+1/2}^{n}+V_{\textup{mesh}}^{n} (t-t^n)$. The discontinuity is not located at the same place for the variables $v$ and $w$, therefore we compute two different crossing times
$$
T_{j+1/2}^{w, n}= \frac{ \Delta x - d_{j-1}^{n,w}}{s_{j-1}^{n}-V_{\textup{mesh}}^{n}} 
\qquad    \text{ and } \qquad T_{j+1/2}^{w, n}= \frac{\Delta x -d_{j-1}^{n,w}}{s_{j-1}^{n}-V_{\textup{mesh}}^{n}}   \qquad \text{ if } \ V_{\textup{mesh}}^{n}<0.
$$
Once again, Condition~\eqref{eq:Vmesh} insures that the flux passing through the $j+1/2$ interface only comes from waves arriving from the cell $j$. It can be clearly seen on Figure~\ref{FMM} that it is piecewise constant, given by
\begin{equation}  \label{eq:fj+1/2} 
\begin{cases}
  f_{j+1/2}^{n,v}= & (-\sigma(w_{j,R}^{n})-V_{\textup{mesh}}^{n} v_{j,R}^{n}) \min(\Delta t^{n}, T_{j+1/2}^{n,v}) \\
		&  +  (-\sigma(w_{j,L}^{n})-V_{\textup{mesh}}^{n} v_{j,L}^{n}) (\Delta t^{n}-\min(\Delta t^{n}, T_{j+1/2}^{n,v})) \\
  f_{j+1/2}^{n,w}= & (-v_{j,R}^{n}-V_{\textup{mesh}}^{n} w_{j,R}^{n}) \min(\Delta t^{n}, T_{j+1/2}^{n,w}) \\
		&  +  (-v_{j,L}^{n}-V_{\textup{mesh}}^{n} w_{j,L}^{n}) (\Delta t^{n}-\min(\Delta t^{n}, T_{j+1/2}^{n,w}))
\end{cases}
\qquad \text{ if } \ V_{\textup{mesh}}^{n}<0
\end{equation}
where we denote by $\Delta t^{n}$ the $n$-th time step $t^{n+1}-t^{n}$. 

When $V_{\textup{mesh}}^{n}$ is negative, the reasoning is the same but the shock reconstructed in the $j$-th cell now crosses the left interface $x=x_{j-1/2}^{n}+V_{\textup{mesh}}^{n} (t-t^n)$. The crossing times are given by
$$
T_{j-1/2}^{w, n}= \frac{d_{j}^{n,w}}{V_{\textup{mesh}}^{n}-s_{j}^{n}} 
\qquad    \text{ and } \qquad T_{j-1/2}^{w, n}= \frac{d_{j}^{n,w}}{V_{\textup{mesh}}^{n}-s_{j}^{n}}   \qquad \text{ if } \ V_{\textup{mesh}}^{n}>0.
$$
and the fluxes are
\begin{equation}  \label{eq:fj-1/2}
\begin{cases}
  f_{j-1/2}^{n,v}= & (-\sigma(w_{j,L}^{n})-V_{\textup{mesh}}^{n} v_{j,L}^{n}) \min(\Delta t^{n}, T_{j-1/2}^{n,v}) \\
		&  +  (-\sigma(w_{j,R}^{n})-V_{\textup{mesh}}^{n} v_{j,R}^{n}) (\Delta t^{n}-\min(\Delta t^{n}, T_{j-1/2}^{n,v})) \\
  f_{j-1/2}^{n,w}= & (-v_{j,L}^{n}-V_{\textup{mesh}}^{n} w_{j,L}^{n}) \min(\Delta t^{n}, T_{j-1/2}^{n,w}) \\
		&  +  (-v_{j,R}^{n}-V_{\textup{mesh}}^{n} w_{j,R}^{n}) (\Delta t^{n}-\min(\Delta t^{n}, T_{j-1/2}^{n,w}))
\end{cases}
\qquad \text{ if } \ V_{\textup{mesh}}^{n}>0
\end{equation}

\begin{rem}
 When no reconstruction is performed, the fluxes~\eqref{eq:fj-1/2} and~\eqref{eq:fj+1/2} coincide with the staggered Lax--Friedrichs fluxes, i.e., to a simple left or right decentering depending on the sign of $V_{\textup{mesh}}^{n}$. The values assigned to the distances $d_{j}^{n,v}$ and $d_{j}^{n,w}$ and to the speed $s_{j}^{n}$ (and thus to the crossing times $T_{j+1/2}^{n,v}$ and $T_{j+1/2}^{n,w}$) are of no importance. 
\end{rem}
\begin{rem}
 The idea of using a moving mesh to simplify the computation of the fluxes goes back to the Lax--Friedrichs scheme. 
 The Rusanov scheme follows the same idea by using a local mesh speed $V_{\textup{mesh}}^{j,n} $ to make the scheme less diffusive. Higher order extensions based on piecewise polynomial reconstructions are possible, see for example~\cite{NT90}.
\end{rem}

\subsection{Detection of classical shocks}
It is also possible to detect and reconstruct shocks in which $w$ does not change sign. Those shocks are always classical. Their detection is based on the following proposition.
\begin{prop} \label{PDetectC}
 If the states $(v_{L},w_{L})$ and $(v_{R},w_{R})$ are linked by a $1$-shock in which $w$ does not change sign, then either
 $$w_{R}<w_{L}\leq 0 \ \text{ and } \ v_{R}<v_{L} $$
 or 
 $$ 0 \leq w_{L}< w_{R} \ \text{ and } \ v_{R}>v_{L}. $$
  If they are linked by a $2$-shock in which $w$ does not change sign, then either
$$w_{L}<w_{R} \leq 0 \ \text{ and } \ v_{R}<v_{L} $$
or
 $$ 0 \leq w_{R}< w_{L} \ \text{ and } \ v_{R}>v_{L}.$$
\end{prop}
There is no conflict with the detection of nonclassical shocks of Proposition~\ref{PDetectPT}, and we can straightforwardly extend Definitions~\ref{DDesired} and~\ref{DRecStates} to take into account those shocks.

\section{Exact approximation of isolated nonclassical shocks}
The aim of this section is to prove that the scheme described above is exact when the initial data is an isolated nonclassical shock, i.e.~\eqref{eq:RPElasto} with left state $(v_{L},w_{L})$ and right state $(v_{R},w_{R})$ verifying the Rankine--Hugoniot relation~\eqref{eq:RHElasto} and constrained by the kinetic relation~\eqref{eq:KR}. We recall that in that case the exact solution is
 $$ 
\begin{cases}
 v^{\textup{exa}}(t,x)= v_{L} \mathbf{1}_{x<(-1)^{i} s(w_{L},w_{R})t} + v_{R} \mathbf{1}_{x> (-1)^{i}s(w_{L},w_{R})t} \\
 w^{\textup{exa}}(t,x)= w_{L} \mathbf{1}_{x< (-1)^{i} s(w_{L},w_{R})t} + w_{R} \mathbf{1}_{x>(-1)^{i} s(w_{L},w_{R})t} 
\end{cases}
 $$
 with $i=1$ for a $1$-shock and $i=2$ for a $2$-shock.

\begin{thm} \label{TExa} The scheme described in the previous section is exact when the initial data is a single nonclassical shock. In other words, the numerical solution is the $L^{1}$-projection on the mesh of the exact solution at time $t^{n}$: for all $n \geq 0$ and for all $j \in \Z$,
$$v_{j}^{n}=\frac{1}{\Delta x} \int_{x_{j-1/2}^{n}}^{x_{j+1/2}^{n}} v^{\textup{exa}}(t^{n},x) \, dx \ \text{ and } \ w_{j}^{n}= \frac{1}{\Delta x} \int_{x_{j-1/2}^{n}}^{x_{j+1/2}^{n}} w^{\textup{exa}}(t^{n},x) \, dx .$$
\end{thm}

\begin{proof} We prove the result for a $1$-nonclassical shock such that $w_{L}<0<w_{R}$, the other cases being exactly similar. With the initial sampling~\eqref{eq:inisamp}, the property holds true at time~$t^{0}$. Suppose that is holds true at some time $t^{n}$. At this time, we renumbered the cells in a way that the discontinuity lies inside the cell numbered $0$, and we denote by $\delta$ its distance to $x_{-1/2}^{n}$. We have
  $$ v_{j}^{n}= 
\begin{cases}
 v_{L} & \text{ if } j<0, \\
 \frac{\delta}{\Delta x} v_{L}  + \frac{\Delta x -\delta}{\Delta x} v_{R} & \text{ if } j=0, \\
 v_{R} & \text{ if } j>1, \\
\end{cases}
\ \text{ and} \
 w_{j}^{n}= 
\begin{cases}
 w_{L} & \text{ if } j<0, \\
 \frac{\delta}{\Delta x} w_{L}  + \frac{\Delta x -\delta}{\Delta x} w_{R} & \text{ if } j=0, \\
 w_{R} & \text{ if } j>1. \\
\end{cases}
 $$
Note that as $w_{L}<w_{R}$ and $v_{R}>v_{L}$, we have $w_{-1}^{n}<w_{0}^{n}<w_{1}^{n}$ and   $v_{-1}^{n}<v_{0}^{n}<v_{1}^{n}$. The detection step of Proposition~\ref{PDetectPT} detects a $1$-nonclassical shock in cell $0$. By Proposition~\ref{p:recexa}, the nonclassical shock is reconstructed in that cell, and it is be placed exactly at the right position in the reconstruction step:
$$(v_{0,L}^{n}, w_{0,L}^{n})= (v_{L}, w_{L}), \quad (v_{0,R}^{n}, w_{j,R}^{n})=(v_{R}, w_{R})  \quad \text{ and } \quad d_{j}^{0,v}=d_{j}^{0,w}= \delta.$$
If no reconstruction is performed in the other cells, our scheme gives the correct solution at time $t^{n+1}$ by construction. This is easy to check, but a little tedious, and we refer the reader to~\cite{A14} for a detailed proof.  

Suppose that $w_{0}^{n}>0$. Then by Proposition~\ref{PDetectC}, no classical shock is detected in cell $1$, and by Proposition~\ref{PDetectPT}, a $1$-nonclassical shock is detected in cell $-1$. Let us focus on the solution Riemann problem between the state $(v_{L},w_{L})$ in cell $-2$ and the state $(v_{0}^{n}, w_{0}^{n})$ in cell $0$. The notation are recalled on Figure~\ref{FRec-1}. If  $\bar{w}_{-1,L}^{n} \geq w_{L}$, both $\bar{w}_{-1,L}^{n}$ and $\bar{w}_{-1,R}^{n}$ are larger than $w_{L}$. Thus it is be impossible to have $d_{-1}^{n,w} \in (0, \Delta x)$ and no reconstruction occurs in cell $-1$. 
\begin{psfrags}
 \psfrag{wL}{$w_{L}$}
 \psfrag{wR}{$w_{R}$}
 \psfrag{w0}{$w_{0}^{n}$}
 \psfrag{w-1L}{$\bar{w}_{-1,L}^{n}$}
 \psfrag{w-1R}{$\bar{w}_{-1,R}^{n}$}
 \psfrag{w=0}{$w=0$}
 \psfrag{-1}{$-1$}
 \psfrag{0}{$0$}
 \psfrag{1}{$1$}
\begin{figure}
 \includegraphics[width=0.9\linewidth]{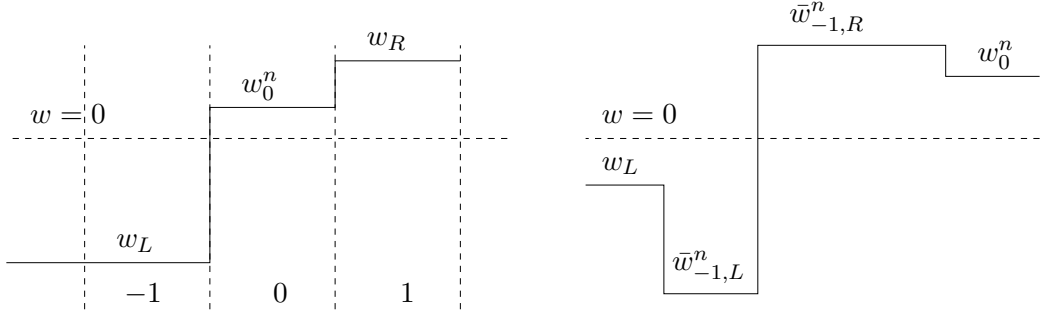}
 \caption{Structure of $w$ when a $1$-phase transition is detected in cell $-1$. In that case, the Riemann solution (on the right) contains three shocks. } \label{FRec-1}
\end{figure}
\end{psfrags}
 Suppose now that $w_{L}$ is larger than $\bar{w}_{-1,L}^{n}$. Then $w_{L}$ and $\bar{w}_{-1,L}^{n}$ are linked by a $1$-classical shock, and $\bar{w}_{-1,L}^{n}$ is linked to $\bar{w}_{-1,R}^{n}$ by a $1$-nonclassical shock. It is is impossible to have $d_{-1}^{n,v}$ in $(0,1)$. Indeed,  the kinetic function $\phi^{\flat, 1}$ decreases, thus
 $$ \bar{w}_{-1,R}^{n}= \phi^{-\flat, 1}( \bar{w}_{-1,L}^{n} ) >  \phi^{-\flat, 1}(w_{L}) = w_{R}  $$
and there is a classical shock between $(\bar{v}_{-1,R}^{n},\bar{w}_{-1,R}^{n})$ and $(v_{R},w_{R})$. The states $(\bar{v}_{-1,R}^{n},\bar{w}_{-1,R}^{n})$ and $(v_{R},w_{R})$ are both on the $1$-wave curve $\mathcal{W}_{1}^{F}(v_{L},w_{L})$. Theorem $4.5$ of~\cite{LT01}  states that along this curve, $v$ is an increasing function of $w$. Therefore $ \bar{v}_{-1,R}^{n}$ is larger than $v_{R}$. On the other hand, the Rankine--Hugoniot relations applied to the $2$-shock yields that $ \bar{v}_{-1,R}^{n}$ is smaller than $v_{-1}^{n}$, which contradicts the fact that $v_{-1}^{n}<v_{R}$.

The case where a $1$-nonclassical shock is detected in cell $1$ is simpler. We have $w_{0}^{n}<0$ and nothing is  is detected on cell $-1$. Moreover, $\bar{w}_{1,L}^{n} \leq \bar{w}_{1,R}^{n}$, thus it is not possible to reconstruct $w$ in a conservative manner if $ \bar{w}_{1,R}^{n} \leq w_{R}$. On the other hand, if $ \bar{w}_{1,R}^{n} > w_{R}$, the solution contains a $2$-classical shock and~\eqref{eq:RH} yields that $\bar{v}_{1,R}^{n} \leq v_{R}$. The Rankine--Hugoniot applied to the $1$-nonclassical shock implies that $\bar{v}_{1,L}^{n} \leq \bar{v}_{1,R}^{n}$. Thus it is impossible that $d_{1}^{n,v}$ belongs to $(0, \Delta x)$.
\end{proof}
\begin{rem}
 It is crucial to require that both $v$ and $w$ are reconstructed in a conservative manner. In the case where the nonclassical shock is detected on cell $1$, it is clear that the proof collapses if we authorize $d_{1}^{n,v}$ outside $(0,\Delta x)$. Moreover if we authorized $d_{1}^{n,v}$ outside $(0,\Delta x)$, the solution of the Riemann problem between $(v_{1}^{n},w_{1}^{n})$ and $(v_{R}, w_{R})$ might contain a $1$-classical shock, a $1$-nonclassical shock and $2$-rarefaction wave, in which case
 $$\bar{v}_{1,L}^{n} \leq v_{R} \leq  \bar{v}_{1,R}^{n} $$
 and a reconstruction is performed in cell $1$. 
\end{rem}

\section{Numerical Simulations} \label{RDouble}
For all the numerical simulations presented below, the stress function is $\sigma(w)= w^{3}+ mw$ with $m \geq 0$. In that case,
$$ \Phi^{\natural}(w)= -\frac{1}{2} w, \qquad \Phi^{\flat}_{\infty}(w)= -w \qquad \text{and} \qquad \Phi^{\sharp}(w,w')=-w-w'. $$
For the kinetic functions we take $\Phi^{\flat,1}(w)=\Phi^{\flat,2}(w)=-\beta w$, where $\beta$ belongs to $[-1/2,1]$. The case $\beta=1/2$ corresponds to the classical solutions ; the choice $\beta=1$ corresponds to the case where the entropy dissipation~\eqref{eq:SEI} is zero across nonclassical choice. It does not fall in the theory of~\cite{LeFloch}, but the Riemann problem can be solved (see~\cite{LT01}) and it is possible to explore that case numerically.

\subsection*{Test $1$: Isolated nonclassical shock}
This test case illustrates Theorem~\ref{TExa}. The initial data is the Riemann problem:
$$ 
\begin{cases}
 v^{0}(x)= -10* \mathbf{1}_{x<0} +110* \mathbf{1}_{x\geq0}, \\
 w^{0}(x)= -6* \mathbf{1}_{x<0} +9* \mathbf{1}_{x\geq0},
\end{cases}
$$
With $m=1$ and $\beta=2/3$, the solution is an isolated nonclassical $1$-shock. On the top of Figure~\ref{FIsolatedNC}, we plot the exact nonclassical solution and the solution given by the reconstruction scheme. As expected they are exactly the same. On the bottom of Figure~\ref{FIsolatedNC}, we plot the solution given by the Godunov scheme based on an exact \emph{nonclassical} Riemann solver. It does not capture the nonclassical solution but the classical one, which in that case in a rarefaction followed by a shock. This is a general phenomenon: usual finite volume schemes are not able to capture nonclassical solutions. Thus in the sequel we used the Glimm scheme to compare our scheme with. The CFL number is set to $0.45$, the final time is $T=0.038$ and the space interval $[-0.5, 0.5]$ is discretized with $200$ cells.
\begin{figure}[htp!]
 \includegraphics[width=\linewidth, clip=true, trim=2cm 0cm 2cm 0cm]{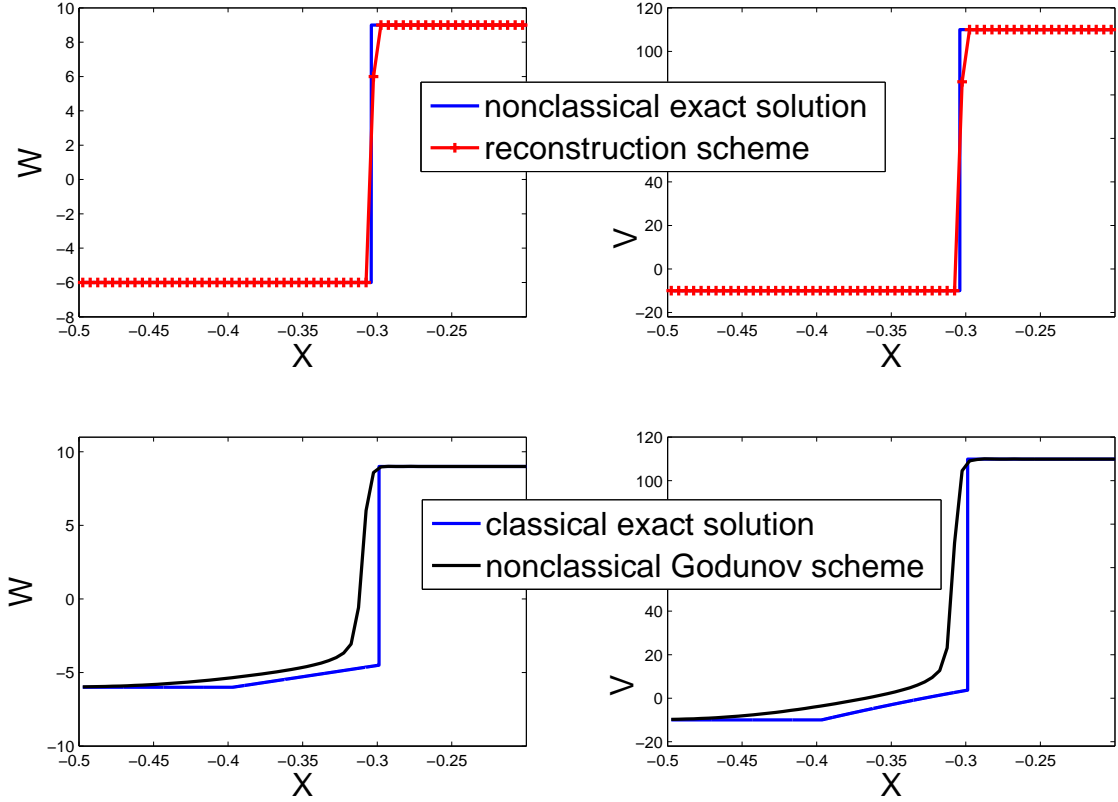}
 \caption{The nonclassical Godunov scheme is unable to capture nonclassical shocks. On the contrary, the reconstruction scheme captures isolated nonclassical shocks exactly.} \label{FIsolatedNC}
 \end{figure}
 
 \subsection*{Interlude: The Glimm scheme}
 In the sequel we use the random sampling method of Glimm~\cite{Glimm65} to compute reference solution. The Glimm scheme is built as follow. Let $(r^{n})_{n \in \N}$ be a sequence of i.i.d. random variables uniformly distributed over $[0,\Delta x]$. Suppose that at time $t^{n}$ a piecewise constant approximation of the solution is given, and denote by $(t,x) \mapsto U^{n}_{\textup{exa}}(t,x)$ the exact solution at time $t \geq 0$ for that initial data.   The numerical solution of Glimm at time $t^{n+1}=t^{n}+ \Delta t^{n}$ is given by
 $$ \forall j \in \Z, \ U_{j}^{n+1}= U^{n}_{\textup{exa}}(\Delta t^{n}, x_{j-1/2}^{n} + r^{n}). $$
 It $\Delta t^{n}$ is small enough, i.e. if it is smaller than the maximum of the waves speed appearing in the Riemann problems multiplied by the cell size $\Delta x$, $U^{n}_{\textup{exa}}$ is just the juxtaposition of  Riemann solutions at each interface. 
 
 The main feature of the Glimm scheme is that is does not introduce any numerical diffusion, in the sense that the shock profiles are not smeared out at all. This is why this scheme has been used in~\cite{CL03} to approximate nonclassical solutions. The two drawbacks are that the scheme is not conservative and that the exact computation of the solution is costly.

\subsection*{Test $2$: A Riemann problem with two nonclassical shocks}
The initial data is now
$$
 \begin{cases}
 v^{0}(x)= 6* \mathbf{1}_{x<0} -10* \mathbf{1}_{x\geq0}, \\
 w^{0}(x)= 1* \mathbf{1}_{x<0} +2* \mathbf{1}_{x\geq0}.
\end{cases}
$$
The exact solution consists in a $1$-classical shock, a $1$-nonclassical shock, a $2$-nonclassical shock and a $2$-rarefaction wave. On Figure~\ref{FFourWaves}, we plot the exact solution and the numerical solution given by the reconstruction schemes. We compare the reconstruction where we detect nonclassical shocks only (cf Proposition~\ref{PDetectPT}), referred to in the legends as $\texttt{RecNC}$, and the reconstruction scheme were the classical shocks are also detected (cf Proposition~\ref{PDetectC}), referred to in the legends as $\texttt{RecNC+C}$. Both of them capture very sharply the nonclassical shocks; the $1$-classical shock is much more diffused when it is not reconstructed, and both scheme behaves in the same way in the rarefaction wave.  The CFL number is $0.45$, the space interval $[-1, 1]$ contains $200$ cells.
\begin{figure}[h!tp]
 \includegraphics[width=\linewidth, clip=true, trim=2cm 0cm 2cm 0cm]{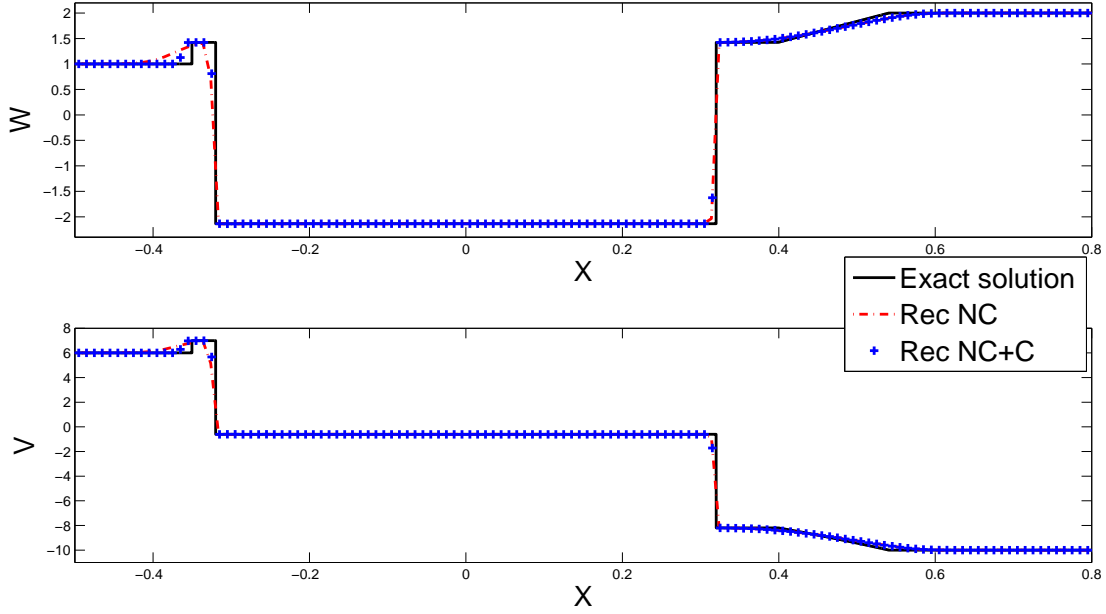}
 \caption{A Riemann problem with two nonclassical shocks at time $0.15$.} \label{FFourWaves}
 \end{figure}
 
\subsection*{Test $3$: Perturbation of a classical shock}
This test case is taken from~\cite{CL03}. We still have $\beta=2/3$ but now $m=2$. The initial data is 
$$ 
\begin{cases}
 v^{0}(x)= 1* \mathbf{1}_{x<0} -11* \mathbf{1}_{x\geq0} ,\\
 w^{0}(x)= (1+\eps)* \mathbf{1}_{x<0} -3* \mathbf{1}_{x\geq0}. 
\end{cases}
$$
with $\eps \in \{0, 0.05, 0.1\}$ When $\eps=0$, the solution is a $1$-classical shock. When $\eps>0$, the classical shock is split into a classical shock followed by a nonclassical shock. Their speeds are different but close to each other. 
We plot the solutions at time $T=0.4$ on Figure~\ref{FClassical}, using $600$ cells per unit interval and a CFL number of $0.45$. The solutions are very well approached when $\eps>0$. When $\eps  = 0$, a spike appears. This spike corresponds to the state linked to $(v_{R},w_{R})$ by a nonclassical shock (on Figure~\ref{FClassical} we see that it has the exact same height as the nonclassical shock appearing for $\eps>0$). Remark that $\eps=0$ is exactly the limit between the two cases in the last point of the enumeration of Proposition~\ref{p:FL1}, for which the speed of the nonclassical and the classical shock coincide (see also Figure~\ref{F:PhiDiese}).

Numerically the mechanism is the following. After one iteration in time, an intermediate value $(v_{i},w_{i})=(\alpha v_{L}+ (1-\alpha) v_{R}, \alpha w_{L}+ (1-\alpha) w_{R})$ is created. At the second iteration in time, the classical shock (in which $w$ changes sign), is perfectly reconstructed in the corresponding cell. However, a $1$-nonclassical shock is also detected by Proposition~\ref{PDetectPT} in the cell just before, and this time the reconstruction succeeds. Thus the scheme is not exact in that case. Note that the proof of Theorem~\ref{TExa} uses the kinetic relation and thus collapses when the reconstructed shock does not verify the kinetic relation (and is hence classical). 

This phenomena does not prevent the scheme from converging. The numerical solution has the same shape than in the case $\eps>0$: a classical shock followed by a nonclassical shock, but this time they have the same speed, thus they remain at the same position.  The spike is only two cells wide when the classical shocks are reconstructed, and hence not diffused. If they are not reconstructed, the spurious classical shock is diffused.

\begin{figure}[h!tp]
 \includegraphics[width=\linewidth, clip=true, trim=2cm 0cm 2cm 0cm]{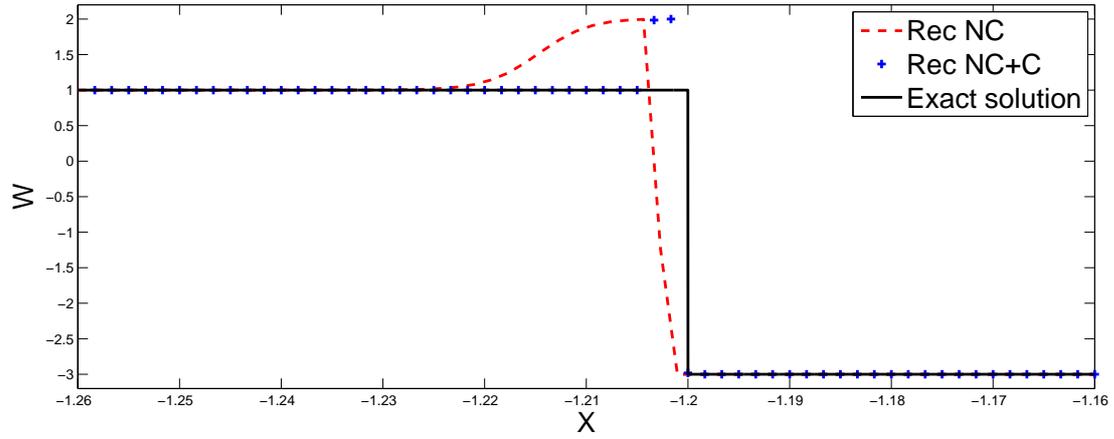}
 \includegraphics[width=\linewidth, clip=true, trim=2cm 0cm 2cm 0cm]{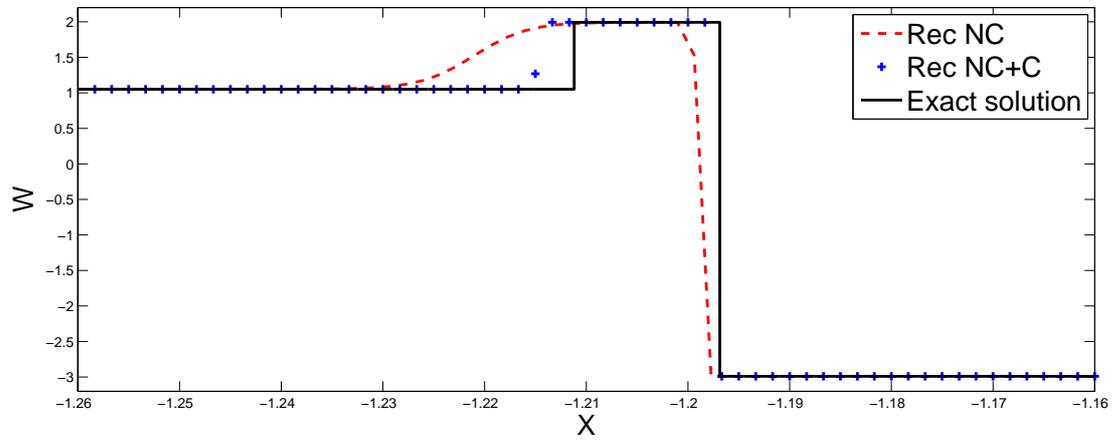}
 \includegraphics[width=\linewidth, clip=true, trim=2cm 0cm 2cm 0cm]{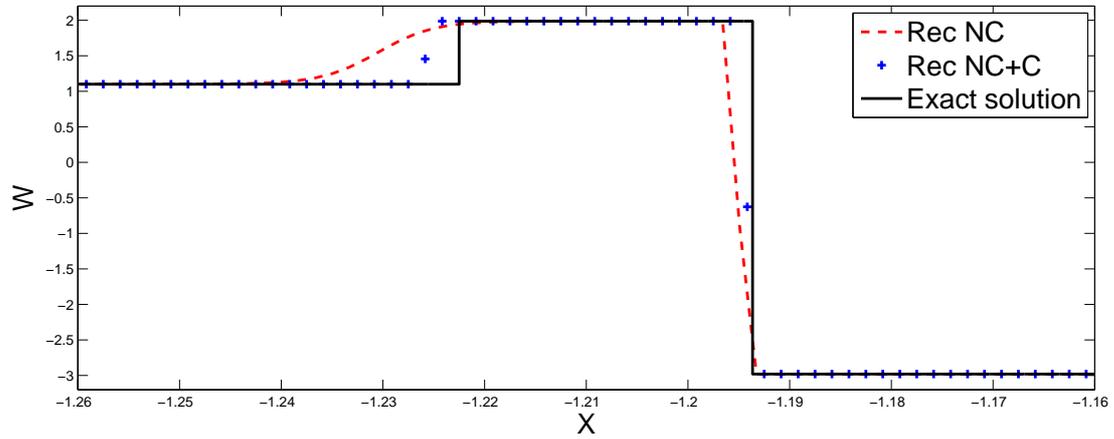}
 \caption{Perturbations of an isolated classical shock, with, from top to bottom, $\eps=0$, $\eps=0.05$ and $\eps=0.1$.} \label{FClassical}
 \end{figure}

In conclusion, this test case shows a limitation of scheme, namely its incapability to approach exactly classical shock in which $w$ changes sign, but also show its ability to capture finely nonclassical solutions, the gap between the shocks being very thin here. In particular, the apparition of the intermediate state is immediate, while it is linked to the ratio of the width of the gap at time $T$ and of the cell size $\Delta x$ when using a random sampling based method like the Glimm scheme.


%

\subsection*{Test $4$: Apparition of nonclassical waves from a smooth initial data}
We now focus on the case of a smooth initial data
$$
 \begin{cases}
 v^{0}(x)= 3 \sin(2 \pi x), \\
 w^{0}(x)= 1+3 \cos(8 \pi x),
\end{cases}
$$
with $\beta=0.95$ and $m=1$ and periodic boundary condition. Nonclassical shocks appear around time $t=0.011$ and then propagate in the solution. On Figures~\ref{F:PT15}, \ref{F:PT60} and~\ref{F:PT100} we compare the solution given by the reconstruction schemes (with or without reconstruction of classical shocks) at time $t=0.015$, $t=0.06$ and $t=1$. The space interval contains $1 \, 000$ cells per unit interval and CFL number is set to $0.45$. The reference solution is given by the Glimm scheme with $8 \, 000$ cells. We see that the reconstruction schemes capture accurately the nonclassical shocks. The result are poorer in smooth areas, because in those regions the reconstruction schemes behave as the Lax--Friedrichs scheme which is quite diffusive. This can be improved by using another scheme on interfaces where no reconstruction is performed. The use of a moving mesh encourages to chose a central scheme like the Nessyahu and Tadmor scheme~\cite{NT90}. This scheme is second order accurate in smooth regions of the solutions. It is both easy to implement and fast, as it does not use any information on the Riemann problems.
\begin{figure}[h!tp]
 \includegraphics[width=\linewidth, clip=true, trim=2cm 1cm 2cm 1cm]{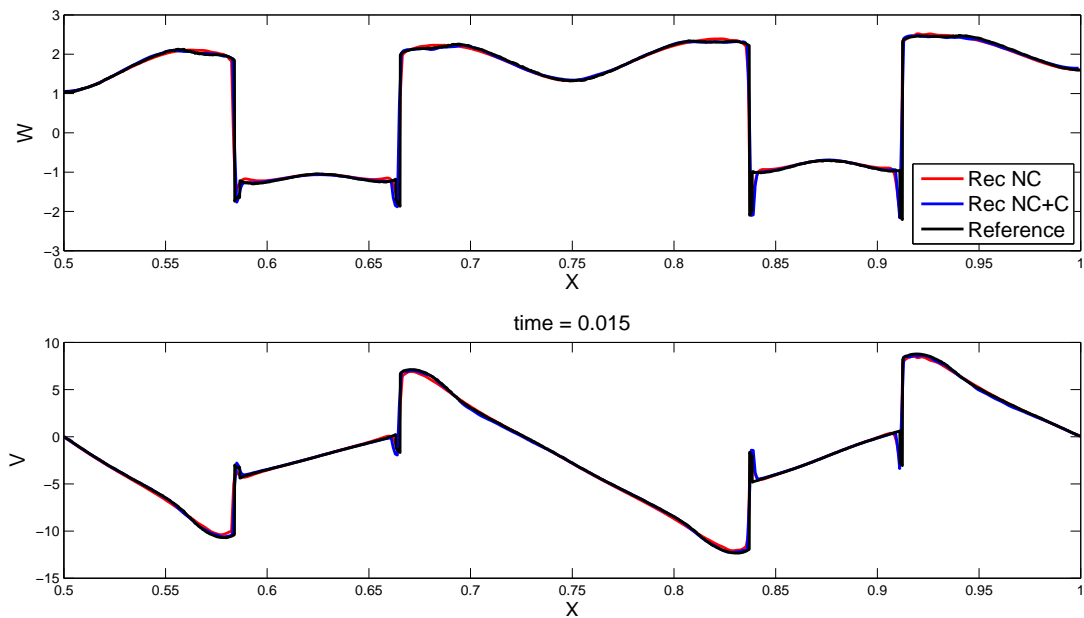}
 \caption{Solution of test $4$ at time $t=0.015$. } \label{F:PT15}
\end{figure}
\begin{figure}[h!tp]
 \includegraphics[width=\linewidth, clip=true, trim=2cm 1cm 2cm 1cm]{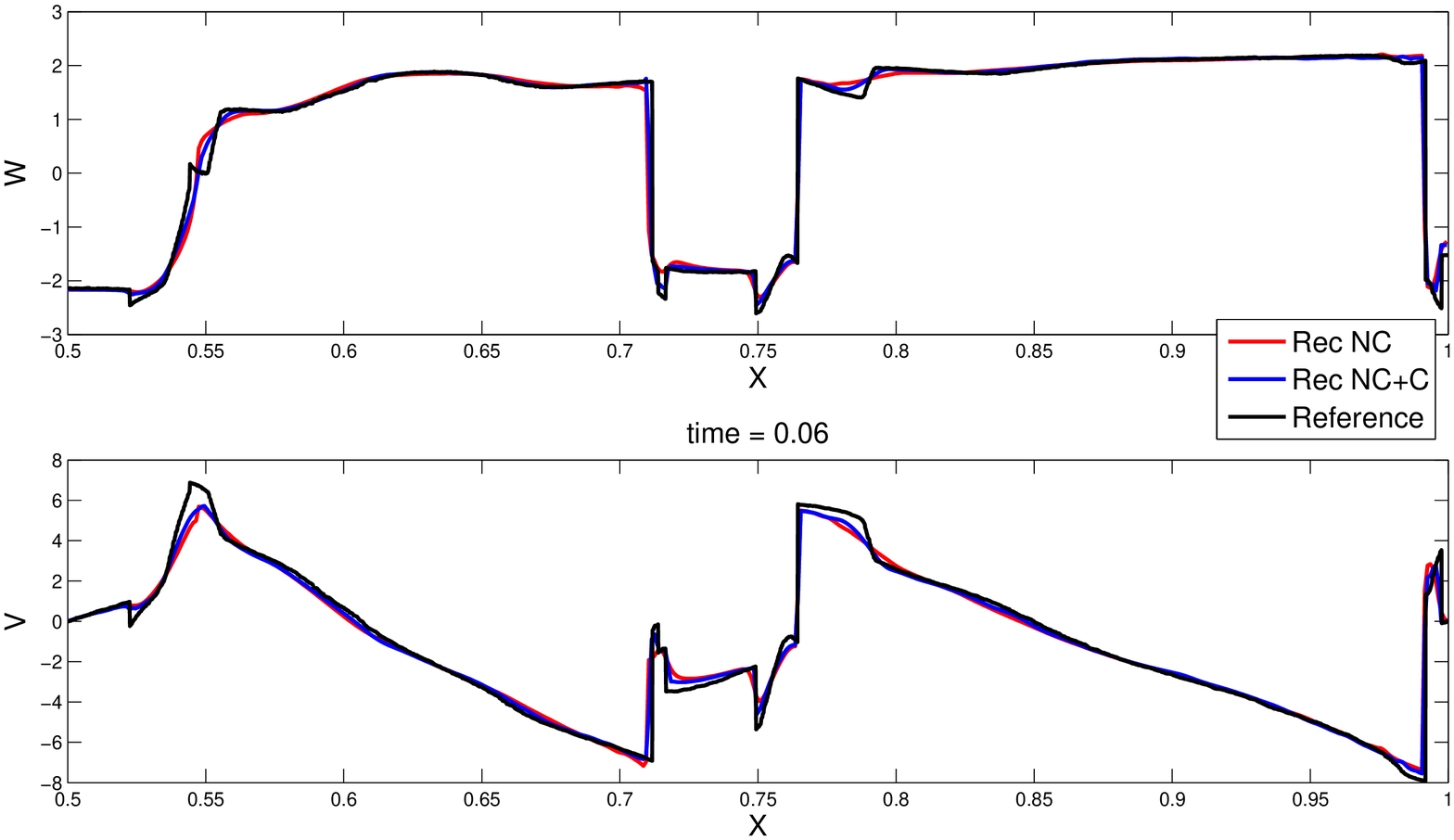}
 \caption{Solution of test $4$ at time $t=0.06$. } \label{F:PT60}
\end{figure}
\begin{figure}[h!tp]
 \includegraphics[width=\linewidth, clip=true, trim=2cm 1cm 2cm 1cm]{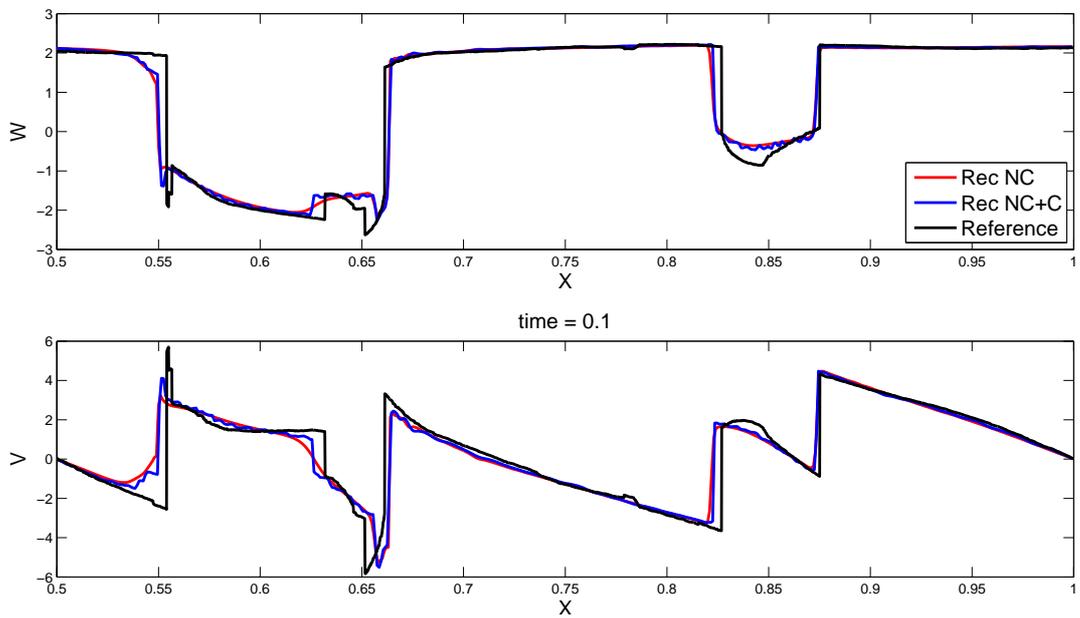}
 \caption{Solution of test $4$ at time $t=0.1$. } \label{F:PT100}
\end{figure}

\subsection*{Test $5$: Long time simulation in the maximal dissipative case}
We reproduce here the test case of~\cite{CL03} for which $\beta=1$ and the final time is large. The case $\beta=1$ corresponds to the case where the entropy dissipation is zero through nonclassical shocks. It is a limit case in the theoretical framework of~\cite{LeFloch}. As the Riemann solver is perfectly defined for $\beta=1$, it can be explored numerically.

The initial data is $1$-periodic with
$$ (v^{0}(x), w^{0}(x))=
\begin{cases}
 (0.3, 0.4) & \text{ if } 0 \leq x < 0.3, \\
 (0.15, -0.2) & \text{ if } 0.3 \leq x < 0.3+2/3, \\
 (0.1, 0.4) & \text{ if } 0.3 +2/3 \leq x < 1.
\end{cases}
$$
Its mean value is null. The parameter $m$ is fixed to $0.05$. On Figure~\ref{F:LFT40}, we plot the solution at time $t=40$ with $2 \, 000$ and $8 \, 000$ points per unit interval. We can see that at that time $w$ changes sign three times, instead of two in the initial solution, and does not converge to zero.
\begin{figure}[h!tp]
 \includegraphics[width=\linewidth, clip=true, trim=2cm 8.4cm 2cm 1cm]{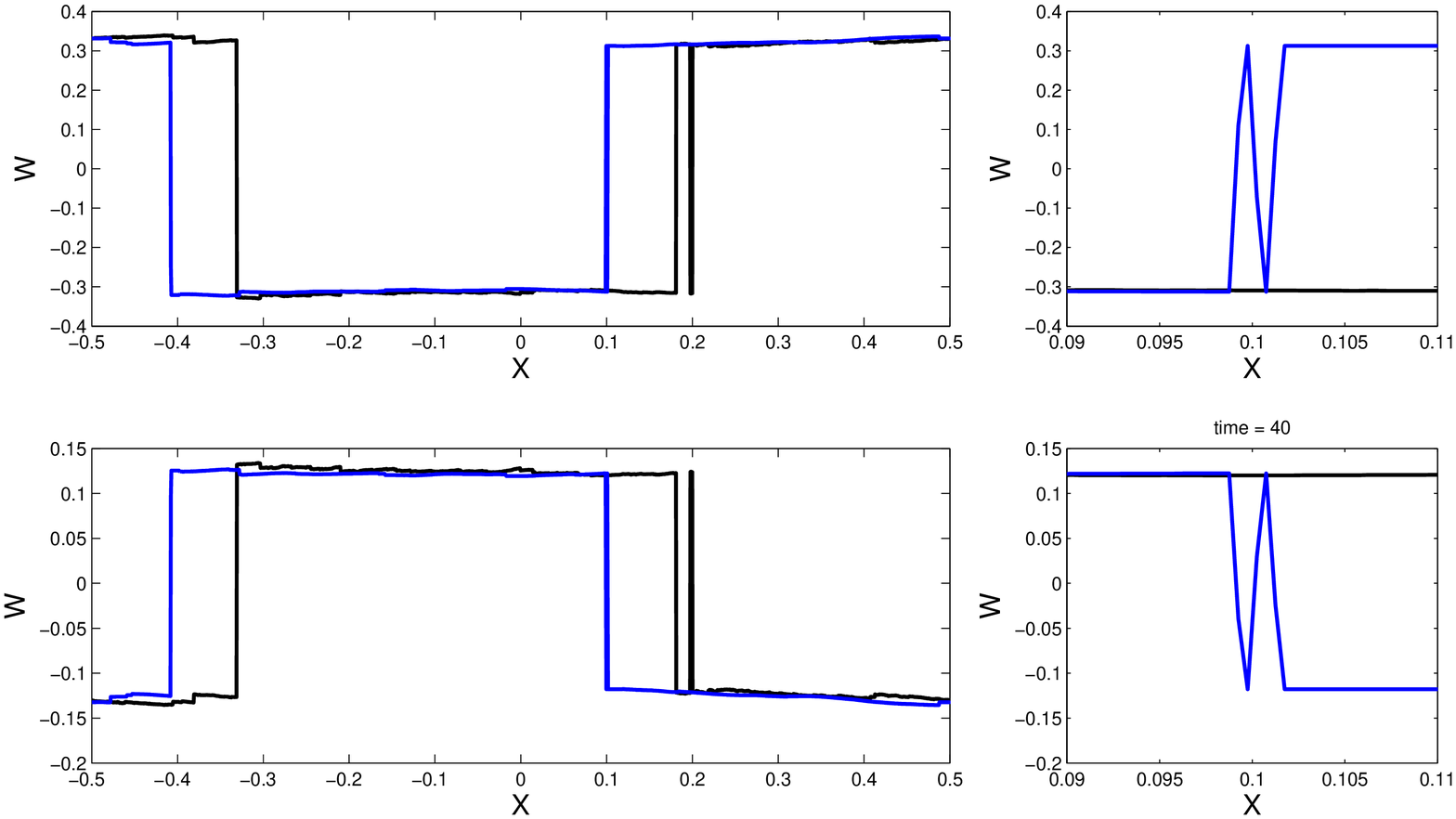}
 \includegraphics[width=\linewidth, clip=true, trim=2cm 8.4cm 2cm 1cm]{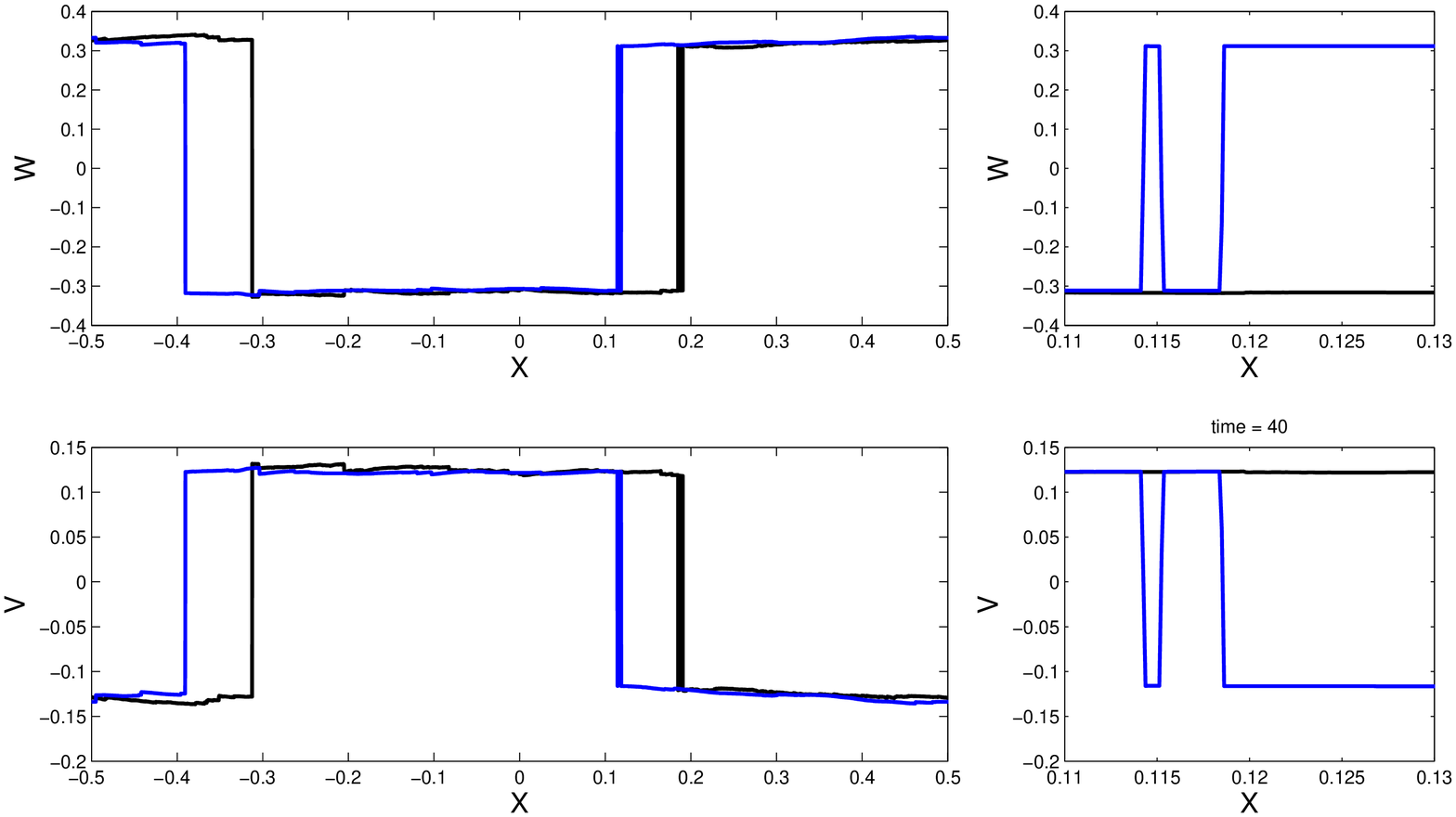}
 \caption{Solution of test $5$ at time $t=40$ (the shape of $v$ is similar) } \label{F:LFT40}
\end{figure}
The position of the nonclassical shocks are different with the two schemes. Let us recall that the Glimm scheme is based on a random sampling of the solution, thus two realizations of the same test give different results. On Figure~\ref{F:HistPos}, we plot the histogram of the first nonclassical shock position at time $t=20$ for $100$ independent realizations of the Glimm scheme with $2 \, 000$ cells (bottom) and the comparison of the two schemes at the same time (with $8 \, 000$ points). Moreover, the structure with $2$ nonclassical shocks very close to each other only appears in $31$ out of those $100$ realizations. Its indicates that the width of this structure is small (and probably smaller than the one appearing in the realization of~\cite{CL03}).
\begin{figure}[h!tp]
 \includegraphics[width=\linewidth, clip=true, trim=2cm 0cm 2cm 1cm]{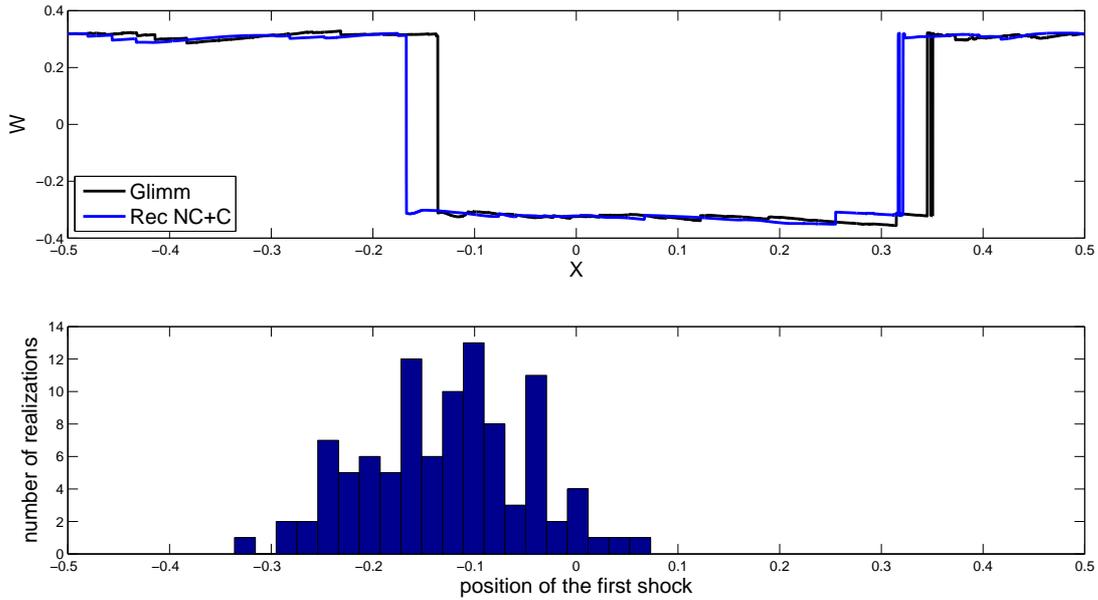}
 \caption{Solution of test $5$ at time $t=40$ (the shape of $v$ is similar) } \label{F:HistPos}
\end{figure}
\phantomsection
\addcontentsline{toc}{section}{References} 

\bibliographystyle{alpha}
\bibliography{biblioReconstructionElasto}

\begin{thebibliography}{CDMG14}

\bibitem[Agu14]{A14}
Nina Aguillon.
\newblock A nondissipative reconstruction scheme for the compressible euler
  equations.
\newblock {\em Preprint,
  \href{http://hal.archives-ouvertes.fr/hal-00967484}{\textit{hal-00967484}}},
  2014.

\bibitem[AK91]{AK91}
Rohan Abeyaratne and James~K. Knowles.
\newblock Kinetic relations and the propagation of phase boundaries in solids.
\newblock {\em Arch. Rational Mech. Anal.}, 114(2):119--154, 1991.

\bibitem[BCLL08]{BCLL08}
Benjamin Boutin, Christophe Chalons, Fr{\'e}d{\'e}ric Lagouti{\`e}re, and
  Philippe~G. LeFloch.
\newblock Convergent and conservative schemes for nonclassical solutions based
  on kinetic relations. {I}.
\newblock {\em Interfaces Free Bound.}, 10(3):399--421, 2008.

\bibitem[CCER12]{CCER12}
Ch. Chalons, F.~Coquel, P.~Engel, and Ch. Rohde.
\newblock Fast relaxation solvers for hyperbolic-elliptic phase transition
  problems.
\newblock {\em SIAM J. Sci. Comput.}, 34(3):A1753--A1776, 2012.

\bibitem[CDMG14]{CDMG}
Christophe Chalons, Maria~Laura Delle~Monache, and Paola Goatin.
\newblock A conservative scheme for non-classical solutions to a strongly
  coupled pde-ode problem.
\newblock {\em Preprint}, 2014.

\bibitem[CG08]{CG08}
Christophe Chalons and Paola Goatin.
\newblock Godunov scheme and sampling technique for computing phase transitions
  in traffic flow modeling.
\newblock {\em Interfaces Free Bound.}, 10(2):197--221, 2008.

\bibitem[CL03]{CL03}
C.~Chalons and P.~G. LeFloch.
\newblock Computing undercompressive waves with the random choice scheme.
  {N}onclassical shock waves.
\newblock {\em Interfaces Free Bound.}, 5(2):129--158, 2003.

\bibitem[Gli65]{Glimm65}
James Glimm.
\newblock Solutions in the large for nonlinear hyperbolic systems of equations.
\newblock {\em Comm. Pure Appl. Math.}, 18:697--715, 1965.

\bibitem[HL97]{HL97}
Brian~T. Hayes and Philippe~G. LeFloch.
\newblock Non-classical shocks and kinetic relations: scalar conservation laws.
\newblock {\em Arch. Rational Mech. Anal.}, 139(1):1--56, 1997.

\bibitem[HL98]{HL98}
Brian~T. Hayes and Philippe~G. Lefloch.
\newblock Nonclassical shocks and kinetic relations: finite difference schemes.
\newblock {\em SIAM J. Numer. Anal.}, 35(6):2169--2194 (electronic), 1998.

\bibitem[HL00]{HL00}
Brian~T. Hayes and Philippe~G. Lefloch.
\newblock Nonclassical shocks and kinetic relations: strictly hyperbolic
  systems.
\newblock {\em SIAM J. Math. Anal.}, 31(5), 2000.

\bibitem[JMS95]{JKS95}
Doug Jacobs, Bill McKinney, and Michael Shearer.
\newblock Travelling wave solutions of the modified {K}orteweg-de
  {V}ries-{B}urgers equation.
\newblock {\em J. Differential Equations}, 116(2):448--467, 1995.

\bibitem[KR10]{KR10}
Frederike Kissling and Christian Rohde.
\newblock The computation of nonclassical shock waves with a heterogeneous
  multiscale method.
\newblock {\em Netw. Heterog. Media}, 5(3):661--674, 2010.

\bibitem[{Lax}57]{Lax57}
Peter~D. {Lax}.
\newblock {Hyperbolic systems of conservation laws. II.}
\newblock {\em {Commun. Pure Appl. Math.}}, 10:537--566, 1957.

\bibitem[LeF02]{LeFloch}
Philippe~G. LeFloch.
\newblock {\em Hyperbolic systems of conservation laws}.
\newblock Lectures in Mathematics ETH Z\"urich. Birkh\"auser Verlag, Basel,
  2002.
\newblock The theory of classical and nonclassical shock waves.

\bibitem[{Liu}74]{Liu74}
Tai-Ping {Liu}.
\newblock {The Riemann problem for general $2 \times 2$ conservation laws.}
\newblock {\em {Trans. Am. Math. Soc.}}, 199:89--112, 1974.

\bibitem[LMR02]{LMR02}
P.~G. Lefloch, J.~M. Mercier, and C.~Rohde.
\newblock Fully discrete, entropy conservative schemes of arbitrary order.
\newblock {\em SIAM J. Numer. Anal.}, 40(5):1968--1992 (electronic), 2002.

\bibitem[LR00]{LR00}
Philippe~G. Lefloch and Christian Rohde.
\newblock High-order schemes, entropy inequalities, and nonclassical shocks.
\newblock {\em SIAM J. Numer. Anal.}, 37(6):2023--2060 (electronic), 2000.

\bibitem[LT00]{LT00}
Philippe~G. LeFloch and Mai~Duc Thanh.
\newblock Nonclassical {R}iemann solvers and kinetic relations. {III}. {A}
  nonconvex hyperbolic model for van der {W}aals fluids.
\newblock {\em Electron. J. Differential Equations}, pages No. 72, 19 pp.
  (electronic), 2000.

\bibitem[LT01]{LT01}
Philippe~G. LeFloch and Mai~Duc Thanh.
\newblock Nonclassical {R}iemann solvers and kinetic relations. {I}. {A}
  nonconvex hyperbolic model of phase transitions.
\newblock {\em Z. angew. Math. Phys.}, 52, 2001.

\bibitem[LT02]{LT02}
Philippe~G. LeFloch and Mai~Duc Thanh.
\newblock Non-classical {R}iemann solvers and kinetic relations. {II}. {A}n
  hyperbolic-elliptic model of phase-transition dynamics.
\newblock {\em Proc. Roy. Soc. Edinburgh Sect. A}, 132(1):181--219, 2002.

\bibitem[NT90]{NT90}
Haim Nessyahu and Eitan Tadmor.
\newblock Nonoscillatory central differencing for hyperbolic conservation laws.
\newblock {\em J. Comput. Phys.}, 87(2):408--463, 1990.

\bibitem[Ole57]{O57}
O.~A. Ole{\u\i}nik.
\newblock On the uniqueness of the generalized solution of the {C}auchy problem
  for a non-linear system of equations occurring in mechanics.
\newblock {\em Uspehi Mat. Nauk (N.S.)}, 12(6(78)):169--176, 1957.

\bibitem[Per11]{P11}
Vincent Perrier.
\newblock A conservative method for the simulation of the isothermal {E}uler
  system with the van-der-{W}aals equation of state.
\newblock {\em J. Sci. Comput.}, 48(1-3):296--303, 2011.

\bibitem[Sle89]{S89}
M.~Slemrod.
\newblock A limiting ``viscosity'' approach to the {R}iemann problem for
  materials exhibiting change of phase.
\newblock {\em Arch. Rational Mech. Anal.}, 105(4):327--365, 1989.

\bibitem[SY95]{SY95}
Michael Shearer and Yadong Yang.
\newblock The {R}iemann problem for a system of conservation laws of mixed type
  with a cubic nonlinearity.
\newblock {\em Proc. Roy. Soc. Edinburgh Sect. A}, 125(4):675--699, 1995.

\bibitem[vL97]{VL79}
Bram van Leer.
\newblock Towards the ultimate conservative difference scheme. {V}. {A}
  second-order sequel to {G}odunov's method.
\newblock {\em J. Comput. Phys.}, 135(2):227--248, 1997.

\end{thebibliography}
\end{document}